%% file: paper.tex
\documentclass[a4paper,reqno]{amsart}

\usepackage[foot]{amsaddr}
\usepackage[english]{babel}

\usepackage[utf8]{inputenc}
\usepackage{csquotes}
\usepackage{amsmath,amssymb,amsfonts,amsthm}
\usepackage{tikz}
\tikzset{>=stealth}
\usetikzlibrary{shapes,arrows,shadows,snakes}
\usepackage{graphicx}
\usepackage{enumerate}
\usepackage{booktabs}
\usepackage{colortbl}
\usepackage{siunitx}
\usepackage{subcaption}
\usepackage{bm}
\usepackage[colorlinks=true,citecolor=blue]{hyperref}
\usepackage[
maxbibnames=99,
maxcitenames=2,
style=numeric,
backend=bibtex,
sorting=nyt,
sortcites,
giveninits=false,
doi=false,
url=false,
natbib=true]{biblatex}

\numberwithin{equation}{section}

\theoremstyle{plain}
\newtheorem{theorem}{Theorem}[section]
\newtheorem{lemma}{Lemma}[section]

\newtheorem{proposition}{Proposition}

\newtheorem{conjecture}{Conjecture}

\theoremstyle{remark}
\newtheorem{remark}{Remark}

\theoremstyle{definition}
\newtheorem{definition}{Definition}
\newtheorem{example}{Example}

\newcommand{\reals}{\mathbb{R}}

\newcommand{\vect}[1]{\bm{#1}}
\newcommand{\abs}[1]{\left\lvert#1\right\rvert}

\newcommand{\fBeta}{\mathcal{F}_\beta}
\newcommand{\hBeta}{\mathcal{H}_\beta}
\newcommand{\whBeta}{\mathcal{WH}_\beta}
\newcommand{\nfBeta}{\mathcal{P}_\beta}
\newcommand{\ini}{\mathcal{N}^{-}}
\newcommand{\outi}{\mathcal{N}^{+}}
\newcommand{\nodeset}{\ensuremath{V}}
\newcommand{\arcset}{\ensuremath{E}}

\let\leq\leqslant
\let\geq\geqslant

\title{Feasible bases for a polytope related to the Hamilton cycle problem}
\author[T. Kalinowski]{Thomas Kalinowski$^{1,2}$}
\author[S. Mohammadian]{Sogol Mohammadian$^2$}
\address{$^1$School of Science and Technology, University of New England, NSW, Australia}
\address{$^2$School of Mathematical and Physical Sciences, University of Newcastle, NSW,
	Australia}

\email[T.~Kalinowski]{t.kalinow@une.edu.au}
\email[S.~Mohammadian]{sogol.mohammadian@newcastle.edu.au}

\date{}

\begin{document}
	
\begin{abstract} 
  We study a certain polytope depending on a graph $G$ and a parameter $\beta\in(0,1)$ which arises
  from embedding the Hamiltonian cycle problem in a discounted Markov decision
  process. \citet{Eshragh2018} conjectured a lower bound on the proportion of feasible bases
  corresponding to Hamiltonian cycles in the set of all feasible bases. We make progress towards a
  proof of the conjecture by proving results about the structure of feasible bases. In particular,
  we prove three main results: (1) the set of feasible bases is independent of the parameter $\beta$
  when the parameter is close to 1, (2) the polytope can be interpreted as a generalized network
  flow polytope and (3) we deduce a combinatorial interpretation of the feasible
  bases. We also provide a full characterization for a special class of feasible bases, and we
  apply this to provide some computational support for the conjecture.
\end{abstract}
	
	
\maketitle

\input{1_introduction}
\input{2_background}
\input{3_general}
\input{4_class_of_bases}

\section{Summary and conclusion}\label{sec:conclusion}
In this paper we continue to study the polytope $\whBeta(G)$ which was introduced
by~\citet{eshragh2011hybrid} in connection with a random walk based approach to the Hamilton Cycle
Problem, and investigated further in~\cite{eshragh2011hamiltonian,Eshragh2018}. Two ingredients are
needed in order to make this approach work: (1) if the input graph is Hamiltonian then there need to
be sufficiently many feasible bases corresponding to Hamiltonian cycles, and (2) the graph of the
polytope needs to have good mixing properties. In order to make progress towards establishing these
two properties, a good understanding of the combinatorial structure of $\whBeta(G)$ is needed, and
in this paper we present significant results in this direction:
\begin{enumerate}
\item The set of feasible bases does not depend on the value of $\beta$ as long as it is
  sufficiently close to 1.
\item $\whBeta(G)$ can be interpreted as a generalized network flow polytope, and this
  interpretation leads to a nice graph theoretical interpretation of the structure of feasible bases.
\item For a special class of bases, we prove a complete characterization of the feasible bases.
\item We illustrate the characterization of feasible bases for small values of $n$ and present 
computational results supporting Conjecture~\ref{con:EFKM}.
\end{enumerate}

\subsection*{Acknowledgement} We thank Hamish Waterer for pointing out the connection to generalized
network flows.

\printbibliography

\end{document}

%% file: 1_introduction.tex
\section{Introduction}\label{sec:intro}
The \emph{Hamilton Cycle Problem} (HCP) is one of the classical problems in combinatorics. Given
a graph~$G$, the problem is to decide if $G$ contains a cycle that visits each node exactly
once. Cycles that pass through every node of a graph exactly once are called \emph{Hamilton
  cycles}. If a graph contains at least one Hamilton cycle, then it is called
\emph{Hamiltonian}. Otherwise, it is \emph{non-Hamiltonian}. The HCP is NP-complete even for planar
graphs with maximum degree three for undirected graphs, and maximum degree two for directed
graphs~\cite{garey1976planar,Plesnik1979}, so it is unlikely that there is an exact algorithm which
terminates in polynomial time and solves the problem in general.

The Traveling Salesman Problem (TSP) asks for a Hamilton cycle of minimum weight in an arc-weighted
graph, and therefore the HCP is a special case of the TSP where all the weights are either zero or
one. Due to its importance in many applications and its rich mathematical structure the TSP has
attracted the attention of many researchers. In particular, it has been one of the major driving
forces for the development of polyhedral techniques in combinatorial
optimization (see~\cite{applegate2011traveling} for a nice overview). The underlying idea is to identify a subset of the arc
set of a graph on $n$ vertices with its characteristic vector in $\reals^{n(n-1)/2}$. Then the TSP
is asking for the minimum of a linear function over the convex hull of the set of  Hamilton cycles,
and this convex hull is known as the traveling salesman polytope. 

In 1994, \citet{filar1994hamiltonian} proposed a new approach to the HCP, based on the theory of
\emph{Markov Decision Processes} (MDPs).  An MDP comprises a state space, an action space,
transition probabilities between states (which depend on the actions taken by the decision maker)
and a reward function. In the basic setting, the decision maker takes an action, receives a reward
from the environment, and the environment changes its state. Next, the decision maker identifies the
state of the environment, takes a further action, obtains a reward, and so forth. The state
transitions are probabilistic, and depend solely on the state and the action taken by the decision
maker. The reward obtained by the decision maker depends on the action taken, and on the current
state of the environment. The decision maker's actions in each environmental state are prescribed by
a policy. The model introduced by \citet{filar1994hamiltonian} initiated a new line of research and
has attracted growing attention (see, for instance,
\cite{avrachenkov2016transition,Borkar2011,ejov2004interior,ejov2008determinants,ejov2009refined,Ejov2011,feinberg2000constrained,Filar2015,litvak2009markov}). In
particular, \citet{feinberg2000constrained} investigated the relationship between the HCP and
\emph{discounted} MDPs. In discounted MDPs, a discount factor $\beta \in (0,1)$, which represents
the difference in importance between future and present rewards, is used to discount
rewards. Feinberg presented a polytope, which we shall refer to as $\fBeta(G)$, constructed from an
input graph $G$. He showed that the Hamilton cycles of $G$ correspond to certain extreme points of
the polytope $\fBeta(G)$, called \emph{Hamiltonian extreme points}. \citet{ejov2009refined}
described some geometric properties of $\fBeta(G)$ and \citet{eshragh2011hybrid} transformed
$\fBeta(G)$ to a combinatorially equivalent polytope $\hBeta(G)$ to avoid certain numerical
issues. Moreover, they constructed a new polytope $\whBeta(G)$ by adding new linear constraints,
called \emph{wedge constraints}, to the polytope $\hBeta(G)$.

In 2011, \citet{eshragh2011hamiltonian} partitioned the extreme points of $\hBeta(G)$ into five
types, consisting of Hamiltonian extreme points and four types of non-Hamiltonian extreme
points. They showed that for a discount factor $\beta$ sufficiently close to one, the wedge
constraints cut off the non-Hamiltonian extreme points of types 2, 3 and 4, while preserving the
Hamiltonian extreme points. In addition, they proposed to use a random walk on the extreme points
(or on the feasible bases) of the polytopes $\hBeta(G)$ and $\whBeta(G)$ to search for a Hamilton
cycle, and they showed that for a Hamiltonian graph $G$, the random walk algorithm detects a
Hamiltonian extreme point with probability one in a finite number of iterations.
 
For a more precise analysis of the efficiency of this random walk approach it is necessary to
understand the combinatorial structure of the polytopes. More precisely, it is required to analyze
the prevalence of Hamiltonian extreme points within the set of all extreme points, as well as the
mixing properties of the random walk. This was the motivation for the work of \citet{Eshragh2018}
who established results on the combinatorial structure of the polytope~$\hBeta(G)$. They
characterized feasible bases of the polytope $\hBeta(G)$ for a general input graph~$G$, and
determined the expected numbers of different types of feasible bases when the underlying graph is
random. They showed that for a random graph, the number of feasible bases corresponding to
Hamiltonian extreme points of $\hBeta(G)$ is exponentially small compared to the total number of
feasible bases. Moreover, they demonstrated that the wedge constraints eliminate a large number of
non-Hamiltonian feasible bases, and they provided computational evidence for the efficiency of the
random walk on the feasible bases of~$\whBeta(G)$. Based on their computational and theoretical
results, they conjectured that for a random graph on $n$ nodes the ratio between the number of
feasible bases corresponding to Hamilton cycles and the total number of feasible bases is
asymptotically bounded below by $c/n^k$ for some positive constants $c$ and $k$.

In this paper, we continue this line of research by studying the structure of the polytope
$\whBeta(G)$. In particular, we are interested in characterizing feasible bases for this
polytope. In Section~\ref{sec:notation}, we introduce some notation and provide relevant background
from the literature. Section~\ref{sec:general} is about general results on feasible bases for
$\whBeta(G)$. In particular, we prove that there is a constant $\alpha^*=\alpha^*(n)<1$ such that for
all graphs $G$ on $n$ vertices, the set of feasible bases for $\whBeta(G)$ does not depend on
$\beta$ as long as $\alpha^*\leq\beta<1$. Then we establish a close relationship between $\whBeta(G)$
and a generalized network flow polytope for a graph that is obtained from $G$ by splitting each node
into two nodes. The third result of Section~\ref{sec:general} is a proof that any feasible basis
contains a set of $n$ arcs forming a collection of node-disjoint cycles, such that in the
corresponding basic feasible solution, the values of the variables corresponding to these $n$ arcs tend
to $1$ as $\beta\to 1$, while the values of the remaining basic variables tend to $0$. In
Section~\ref{sec:class_of_bases}, we use the relation to the generalized network flow polytope to
give a complete characterization of feasible bases of $\whBeta(K_n)$ which correspond to a
Hamilton cycle in the generalized network flow setting.

%%% Local Variables:
%%% mode: latex
%%% TeX-master: "paper"
%%% End:

%% file: 2_background.tex
\section{Notation and background}\label{sec:notation}
Consider a digraph $G=(\nodeset,\arcset)$ without loops and parallel arcs, where
$\nodeset=[n]=\{1,2,\dots,n\}$ is the set of nodes and $\arcset$ is the set of arcs. Throughout this
paper, $G$ refers to such a digraph on $n$ nodes, unless otherwise stated. For each node
$i\in \nodeset$, the in-neighborhood $\ini(i)$, and the out-neighborhood $\outi(i)$ are the sets
\begin{align*}
  \ini(i) &= \{j\in \nodeset\,:\,(j,i)\in \arcset\},&  \outi(i)&=\{j\in \nodeset\,:\,(i,j)\in \arcset\}.
\end{align*}
Furthermore, we use the following
notations to denote the total inflow and total outflow for node $i$, respectively:
\begin{align*}
\Phi(i)&=\sum_{j\in \ini(i)} x_{ji},& \Psi(i)&=\sum_{j\in \outi(i)} x_{ij}. 
\end{align*}

As indicated earlier, \citet{feinberg2000constrained} defined a polytope depending on the graph $G$,
and showed that finding a Hamilton cycle is equivalent to finding an extreme point of this polytope
whose \emph{support} corresponds to a Hamilton cycle. The support of an extreme point is defined to be the set of its non-zero coordinates. More precisely, he proved the following theorem.
\begin{theorem}[\citet{feinberg2000constrained}]\label{thm:Fienberg_thm}
  Consider a digraph $G=(\nodeset,\arcset)$, a parameter $\beta$ with $0<\beta<1$, and let
  $\fBeta(G)\subseteq\reals^{\lvert E\rvert}$ be the polytope defined by the constraints
  \begin{align}
    \sum_{j \in \outi(1)} y_{1j} - \beta \sum_{j \in \ini(1)} y_{j1} &= 1,\label{eq_f:flow_conservation_1}\\
    \sum_{j \in \outi(i)} y_{ij} - \beta \sum_{j \in \ini(i)} y_{ji} &= 0 && \text{for all } i \in \nodeset \setminus \{1\},\label{eq_f:flow_conservation_i}\\
    \sum_{j \in \outi(1)} y_{1j} & = \frac{1}{1 - \beta^n},\label{eq_f:flow_injection}\\
    y_{ij}&\geq 0 &&\mbox{for all }(i,j) \in \arcset.\label{eq_f:nonnegativity}
  \end{align} 
  The graph $G$ is Hamiltonian if and only if there exists an extreme point of $\fBeta(G)$ which has
  exactly $n$ positive coordinates tracing out a Hamilton cycle in $G$.
\end{theorem}  
\citet{eshragh2011hybrid} modified $\fBeta(G)$ by a coordinate transformation $x_{ij}= (1-\beta^n)y_{ij}$
for all $(i,j) \in \arcset$. The resulting polytope $\hBeta(G) \subseteq \reals^{\lvert \arcset
  \rvert}$ is defined by the constraints
\begin{align}
\sum_{j\in \outi(1)} x_{1j}-\beta\sum_{j\in \ini(1)}
x_{j1} & = 1-\beta^n,  \label{eq_h:flow_conservation_1}\\
\sum_{j\in \outi(i)} x_{ij}-\beta\sum_{j\in \ini(i)}
x_{ji} & = 0 &&\mbox{for all}\ i \in \nodeset \setminus\{1\},  \label{eq_h:flow_conservation_i}\\
\medskip \sum_{j\in \outi(1)}x_{1j} & = 1, \label{eq_h:flow_injection} \\
x_{ij} & \geq 0 &&\mbox{for all}\ (i,j)\in \arcset. \label{eq_h:nonnegativity}
\end{align}
Since values of $\beta$ close to one were shown to be important in
\cite{eshragh2011hamiltonian,eshragh2011hybrid}, this transformation eliminates numerical
instability in~\eqref{eq_f:flow_injection}. The following definition is motivated directly from
Theorem~\ref{thm:Fienberg_thm}.
\begin{definition}\label{def:Ham_EP}
  Let $\vect x$ be an extreme point of the polytope $\mathcal{H}_\beta(G)$. If the positive
  coordinates of $\vect x$ trace out a Hamilton cycle in the graph $G$, $\vect x$ is called a
  Hamiltonian extreme point. Otherwise, it is called a non-Hamiltonian extreme point.
\end{definition}
\citet{eshragh2011hybrid} observed that, if $\vect{x}$ is the Hamiltonian extreme point corresponding
to a Hamilton cycle~$C$ then its components are given by $x_{ij}=\beta^k$ if $(i,j)$ is the
$(k-1)$-th arc in $C$ starting from node 1, and $x_{ij}=0$ if $(i,j)$ is not contained in
$C$. In particular, the Hamiltonian extreme points satisfy the $2(n-1)$ wedge constraints 
\begin{equation}\label{wedge_constraints}
\beta^{n-1} \leq \sum_{j \in \outi(i)} x_{ij} \leq \beta \qquad \mbox{for all } i \in \nodeset \setminus \{1\},  
\end{equation}
which cut off some non-Hamiltonian extreme points. Adding the wedge constraints and introducing
slack variables $y_i$, we obtain the polytope $\whBeta(G)$ described by the following constraints:
\begin{align}
  \sum_{j\in\outi(1)}x_{1j} - \beta\sum_{j\in\ini(1)}x_{j1} & = 1 - \beta^n,\label{eq_wh:flow_extraction}\\
  \sum_{j\in\outi(i)}x_{ij}-\beta\sum_{j\in\ini(i)}x_{ji} &= 0 && \text{for all }i\in \nodeset\setminus\{1\},\label{eq_wh:flow_conservation}\\
  \sum_{j\in\outi(1)}x_{1j} &= 1,\label{eq_wh:flow_injection}\\
  \sum_{j\in\outi(i)}x_{ij}-y_i &= \beta^{n-1} &&\text{for all }i\in\nodeset\setminus\{1\},\label{eq_wh:flow_bounds}\\
  0\leq y_i &\leq \beta-\beta^{n-1} &&\text{for all }i\in\nodeset \setminus\{1\},\label{eq_wh:slacks}\\
  x_{ij} &\geq 0 &&\text{for all }(i,j)\in E. \label{eq_wh:nonnegativity} 
\end{align}
A basis for this polytope can be specified by a triple $(B,L,U)$, where $B$ is the set of basic
variables, and the sets $L$ and $U$ are non-basic $y$-variables at their lower and upper bounds,
respectively. In other words, $L\cup U$ is the partition of the set
$\{i\in V\setminus\{1\}\,:\,y_i\text{ is a non-basic variable}\}$, such that $y_i=0$ for $i\in L$
and $y_i=\beta-\beta^{n-1}$ for $i\in U$. By a slight abuse of notation, we will simply call the
triple $(B,L,U)$ a basis. The set $B$ can be identified with the union $A\cup Y$, where
$A\subseteq E$ is the set of arcs $(i,j)\in E$ such that $x_{ij}$ is a basic variable, and
$Y\subseteq V\setminus\{1\}$ is the set of nodes $i$ such that $y_i$ is a basic variable. A basis
$(B,L,U)$ can then be interpreted as a node-colored digraph on the node set $V$: the arc set is $A$
and the color classes are $\{1\}$, $Y$, $L$ and $U$.  If the unique solution of the system of
equations~\eqref{eq_wh:flow_extraction}--\eqref{eq_wh:flow_bounds} corresponding to $(B,L,U)$
satisfy the lower and upper bound constrains~\eqref{eq_wh:slacks}--\eqref{eq_wh:nonnegativity}, then
the basis $(B,L,U)$ is feasible, otherwise it is infeasible.
\begin{remark}
  For the extreme point $(\vect x,\vect y)$ corresponding to a basis $(B,L,U)$, the total in and
  outflows of the nodes in $L\cup U$ are as follows
  \begin{align}
    \Phi(i) &=\sum_{j\in \ini(i)} x_{ji}=
              \begin{cases}
                \beta^{n-2} &\text{for }i\in L,\\
                1&\text{for }i\in U,
              \end{cases}\label{eq:inflow_LU}\\
    \Psi(i) &=\sum_{j\in \ini(i)} x_{ji}=
              \begin{cases}
                \beta^{n-1} &\text{for }i\in L,\\
                \beta&\text{for }i\in U,
              \end{cases}\label{eq:outflow_LU}
  \end{align}
\end{remark}
\citet{eshragh2011hybrid} introduced the concept of quasi-Hamiltonian extreme points (bases) to
search for Hamilton cycles among the extreme points (or the feasible bases) of the polytope
$\whBeta(G)$.
\begin{definition}\label{def:QHam_EP}
  An extreme point $(\vect x,\vect y)$ of $\whBeta(G)$ is called \emph{quasi-Hamiltonian} if any walk
  $i_1 \rightarrow i_2 \rightarrow \dots\rightarrow i_n \rightarrow i_{n+1}$, where
  $i_{k+1} \in \arg\max_{j \in \mathcal{N^+}(i_k)}\{ x_{i_k j} \}$ for $k = 1,\dots, n$, is a
  Hamilton cycle in $G$. A feasible basis corresponding to a quasi-Hamiltonian extreme point is called
  \emph{quasi-Hamiltonian basis}.
\end{definition}
For a positive integer $n$ and a probability $p$, $0<p<1$ (which may depend on $n$), let
$\bar G_{n,p}$ be the graph on $n$ vertices obtained by adding each arc $(i,j)$ independently with
probability $p$, and then adding the arcs of a randomly chosen Hamilton cycle (this is very similar
to the random graph model studied in~\cite{broder1994finding}). Motivated by a random walk approach
to the HCP for sparse random graphs, the following conjecture was made in~\cite{Eshragh2018}.
\begin{conjecture}\label{con:EFKM}
  There exist positive constants $c$, $\delta$ and $k$ such that for all $\beta\in(1-e^{-cn},1)$,
  with high probability, the expected proportion of feasible bases of $\whBeta(\bar G_{n,p})$ that
  are quasi-Hamiltonian is at least $\delta/n^k$.
\end{conjecture}
Proving this conjecture would be a first step towards a random walk based algorithm for the HCP on
$\bar G_{n,p}$ for very small $p$.

%%% Local Variables:
%%% mode: latex
%%% TeX-master: "paper"
%%% End:

%% file: 3_general.tex
\section{General results}\label{sec:general}
In this section we establish general results about feasible bases of $\whBeta(G)$. In
Subsection~\ref{subsec:ultimate}, we show that the set of feasible bases for $\whBeta(G)$ does not
depend on the parameter $\beta$ for values of $\beta$ sufficiently close to one. Then, in
Subsection~\ref{subsec:gnf}, we demonstrate that $\whBeta(G)$ can be interpreted as a generalized
network flow polytope for a certain network on $2n$ nodes obtained by splitting each node of $G$
into two nodes. The polyhedral theory of the generalized network flow problem implies that the bases
in this setting correspond to subgraphs in which every connected component is an augmented tree,
that is, a graph obtained from adding a single arc to a tree, thus creating a unique cycle. We prove
that the \emph{feasible} bases in the generalized network flow interpretation of $\whBeta(G)$ are
always connected, that is, a feasible basis corresponds to a spanning augmented tree. Finally, in
Subsection~\ref{subsec:thick_and_thin}, we show that in the original graph $G$ every feasible basis
contains a spanning collection of node-disjoint cycles, such that the values of the variables
corresponding to the arcs of these cycles are close to one, while the values of the remaining
variables are close to zero.

\subsection{Ultimate bases}\label{subsec:ultimate}
Let $\Gamma(G,\beta)$ be the set of feasible bases for $\whBeta(G)$. We first show that there exists
$\alpha^*\in (0,1)$ such that $\Gamma(G,\beta_1)=\Gamma(G,\beta_2)$ whenever
$\alpha^*<\beta_1\leq\beta_2<1$, that is, the set $\Gamma(G,\beta)$ does not depend on the parameter
$\beta$, as long as it is sufficiently close to $1$. Consider a triple $(B,L,U)$ with $\abs{B}=2n$,
where $B=A\cup Y$, $A\subseteq E$ and $Y\subseteq V\setminus\{1\}$, and $L\cup U$ is a partition of
the set $V\setminus( \{1\}\cup Y)$. Let $M_B(\beta)$ be the submatrix of the constraint matrix of
the polytope $\whBeta(G)$ corresponding to $B$, set $y_i=0$ for $i\in L$ and $y_i=\beta-\beta^{n-1}$
for $i\in U$, and let $\vect{b}(\beta)$ be the vector that is obtained from the right hand sides of
\eqref{eq_wh:flow_extraction} through~\eqref{eq_wh:flow_bounds}, where for $i\in U$ the
$\beta^{n-1}$ in~\eqref{eq_wh:flow_bounds} is replaced by
$\beta^{n-1}+\left(\beta-\beta^{n-1}\right)=\beta$. The triple $(B,L,U)$ is a feasible basis for
$\whBeta(G)$ if and only if the following two conditions are satisfied:
\begin{enumerate}
\item $M_B(\beta)$ is invertible, that is, $\det(M_B(\beta))\neq 0$, and
\item the $y$-components of the vector $(\vect x,\vect y)=M_B(\beta)^{-1}\vect b(\beta)$
  satisfy~\eqref{eq_wh:slacks}, and the $x$-components satisfy~\eqref{eq_wh:nonnegativity}.
\end{enumerate}
We refer to these conditions as \emph{independence} and \emph{feasibility}, respectively. In the
next two lemmas we show that these two properties do not depend on $\beta$ for $\beta$ sufficiently
close to~$1$. The determinant $\det(M_B(\beta))$ is a polynomial in $\beta$. If this is the zero
polynomial, then the independence condition is not satisfied for any $\beta$. On the other hand, if
the determinant is not the zero polynomial, then $B$ is independent for all $\beta$ sufficiently
close to~$1$. For the formal argument we let $\mathcal B$ be the set of all sets of $2n$ variables
such that the corresponding columns are independent for some $\beta$, that is,
\[\mathcal B=\{B\,:\,B=A\cup Y,\, A\subseteq E,\,Y\subseteq V\setminus\{1\},\, \abs{B}=2n,\,
  \,\det(M_B(\beta)) \not\equiv 0\}.\]
Since $\mathcal B$ is finite, there is an $\alpha\in[0,1)$
such that none of the polynomials $\det M_B(\beta)$, $B\in\mathcal B$, has a root in the open interval
$(\alpha,1)$. This implies the following lemma.
\begin{lemma}\label{lem:bound1_on_beta}
  There exists $\alpha \in [0,1)$ such that for all $\beta > \alpha$, and for all $B=A\cup Y$,
  $A\subseteq E$, $Y\subseteq V\setminus\{1\}$, with $\lvert B \rvert = 2n$, the matrix $M_B(\beta)$
  is invertible for every $B\in\mathcal B$.
\end{lemma}
\begin{proof}
  Define a function $H:\mathcal B\to[0,1)$ as follows:
  \[ H(B) =
    \begin{cases}
      \max\{ \beta\, :\,0\leq\beta< 1\, \det(M_B(\beta)) = 0  \} & \text{if }\det(M_B(\beta)) = 0 \text{ for some }\beta\in(0,1),\\
      0 & \text{otherwise.}
    \end{cases} \]
  Then $\alpha= \max\{ H(B)\,:\, B\in\mathcal B\}$ has the claimed property. It
  follows immediately from the construction that $0\leq\alpha<1$, and that for every $B\in\mathcal
  B$ and every $\beta$ with $\alpha<\beta<1$, $\det(M_B(\beta))\neq 0$.
\end{proof}
The next step is to show that feasibility is also independent of $\beta$ for values sufficiently
close to~$1$.
\begin{lemma}\label{lem:bound2_on_beta}
  There exists $\alpha^*\in(\alpha,1)$ such that for all $\beta \geq \alpha^*$, and for all
  $B\in\mathcal B$, the vector $M_B(\beta)^{-1}\vect b(\beta)$ satisfies~\eqref{eq_wh:slacks}
  and~\eqref{eq_wh:nonnegativity} if and only if $M_B(\alpha^*)^{-1}\vect b(\alpha^*)$
  satisfies~\eqref{eq_wh:slacks} and~\eqref{eq_wh:nonnegativity}.
\end{lemma}
\begin{proof} 
  For $B=A\cup Y\in\mathcal B$, and let $\beta\in(\alpha,1)$, let $\left(\vect x^B(\beta),\vect
    y^B(\beta)\right)=M_B^{-1}(\beta)\vect b(\beta)$. The components of this vector are rational functions
    of $\beta$ with denominator $\det(M_B(\beta))$, say
    \begin{align*}
      x^B_{ij}(\beta) &= \frac{f^B_{ij}(\beta)}{\det(M_B(\beta))} \text{ for }(i,j)\in A,&
      y^B_{i}(\beta) &= \frac{g^B_{i}(\beta)}{\det(M_B(\beta))}\text{ for } i\in Y.
    \end{align*}
    Let $h_i^B(\beta)= \beta- \beta^{n-1}- g^B_{i}(\beta)/\det(M_B(\beta))$. We define
  \[T_{ij}(B)=\max\{\beta\,:\,0<\beta< 1,\,\det(M_B(\beta))f^B_{ij}(\beta) = 0\}\]
  if $\det(M_B(\beta))f^B_{ij}(\beta)$ is not identically zero, but vanishes for some
  $\beta\in(0,1)$, and $T_{ij}(B)=0$ otherwise. Similarly, we define
  \[K_{i}(B)=\max\{\beta\,:\,0<\beta< 1,\,h_i^B(\beta)g^B_{i}(\beta) = 0\}\]
  if $h_i^B(\beta)g^B_{i}(\beta)$ is not identically zero, but vanishes for some
  $\beta\in(0,1)$, and $K_{i}(B)=0$ otherwise. Let
  \begin{align*} 
  \tau = \max_{B} \max_{(i,j) \in A} T_{ij}(B),&& \kappa= \max_{B} \max_{i \in Y} K_{i}(B),
  \end{align*}
  and $\gamma=\max\{\tau,\kappa\}$. Let $\alpha^*=\gamma +(1-\gamma)/2\in(\alpha,1)$. It follows
  from Lemma~\ref{lem:bound1_on_beta} and $\alpha^*\in(\alpha,1)$, that $M_B(\beta)$ is invertible
  for every $B\in\mathcal B$ and every $\beta\in[\alpha^*,1)$. Furthermore, as
  $\alpha^*>\max\{\tau,\kappa\}$, none of the functions $f^B_{ij}(\cdot)$, $g^B_{i}(\cdot)$,
  $h_i^B(\cdot)$ and $\det(A_B(\cdot))$ changes sign on the interval $[\alpha^*,1)$, and this
  implies that $\alpha^*$ has the claimed property.
\end{proof}
Combining Lemmas~\ref{lem:bound1_on_beta} and~\ref{lem:bound2_on_beta} we obtain the following
result which justifies the subsequent definition.
\begin{proposition}\label{prop:ultimate_bases}
  There exists $\alpha^*\geq\alpha$ such that for all $\beta\in[\alpha^*,1)$, $\Gamma(G,\beta) = \Gamma(G,\alpha^*)$.
\end{proposition}
\begin{definition}\label{def:ultimate_bases}
  We call the elements of $\Gamma(G,\beta)$ for $\beta\geq\alpha^*$ \emph{ultimate bases for $G$},
  and we denote the set of ultimate bases for $G$ by $\Gamma(G)$.
\end{definition}
Henceforth, we consider ultimate bases of $\whBeta(G)$ and we set
$\beta= 1 - \delta$ where $\delta$ tends to zero. We omit the argument $\beta$ whenever there
is no danger of confusion.  
\subsection{A generalized network flow formulation}\label{subsec:gnf}
In this subsection, we show that the polytope $\whBeta(G)$ can be interpreted as a generalized network flow (GNF)
polytope and we establish results about feasible bases of $\whBeta(G)$ in the GNF setting. We first
recall some definitions and results pertaining to the GNF polytope, see~\cite[Chapter 15]{ahuja1993network} for
details. Let $X=(V(X),E(X))$ be a digraph with capacities $c_{v,w}$ and positive rational multipliers
$\mu_{vw} > 0$ for each arc $(v,w) \in \arcset(X)$, and demands/supply $b_v$ for each node
$v \in \nodeset(X)$. The GNF polytope for this data is defined by the following constraints:
\begin{align*}
\sum_{w\in \outi(v)} x_{vw} - \sum_{w\in \ini(v)} \mu_{wv} x_{wv}& = b_v && v \in V(X),\\
0 \leq x_{vw} &\leq u_{vw}  && (v,w) \in \arcset(X).
\end{align*}
A possible interpretation is that for every unit of flow that enters arc $(v,w)$ in node $v$,
$\mu_{vw}$ units arrive at node~$w$. The polytope $\whBeta(G)$ given
by~\eqref{eq_wh:flow_extraction} through~\eqref{eq_wh:nonnegativity} is a GNF polytope for the
digraph $G'$ obtained from $G$ by splitting each node $i$ into two nodes $v_i$ and $w_i$, replacing
arcs $(i,j)$ by $(w_i,v_j)$, and adding arcs $(v_i,w_i)$ for $i=2,3,\dots,n$. Then $y_i$ is the flow
on arc $(v_i,w_i)$, the multipliers are equal to $\beta$ for all arcs of the form $(w_i,v_j)$, and
equal to $1$ for all arcs of the form $(v_i,w_i)$. More precisely, the digraph
$G'=(\nodeset',\arcset')$ has node set $V'$ and arc set $E'$ given by
\begin{align*}
  \nodeset' &= \{v_i\,:\,i\in \nodeset\}\cup\{w_i\,:\,i\in \nodeset\}, &
  \arcset' &= \{(w_i,v_j)\,:\,(i,j)\in E\}\cup\{(v_i,w_i)\,:\,i\in\nodeset\setminus\{1\}\}.              
\end{align*}
We denote the two parts in the partition of $E'$ as $E'_1$ and $E'_2$, that is,
$E'_1=\{(w_i,v_j)\,:\,(i,j)\in E\}$ and $E'_2=\{(v_i,w_i)\,:\,i\in\nodeset\setminus\{1\}\}$.
This construction of $G'$ is illustrated in Figure~\ref{fig:ex_G_G'}.
\begin{figure}
  \centering
  \begin{subfigure}[b]{.36\textwidth}
    \centering \vspace{1.2cm}
    \begin{tikzpicture}[scale=1.5,every node/.style={draw,circle,outer sep=1pt,inner
        sep=0pt,minimum size=1.5mm,fill=none}]
      \foreach \i in {1,...,4} {
        \node(v\i) at (-90*\i-135:1) [label={-90*\i-135:$\i$}]{};
      }
      \draw[thick,bend right=15,->] (v1) to (v2);
      \draw[thick,bend right=15,->] (v2) to (v3);
      \draw[thick,bend right=15,->] (v3) to (v2);
      \draw[thick,->] (v4)--(v3);
      \draw[thick,bend right=15,->] (v2) to (v1);
      \draw[thick,->] (v1)--(v3);
      \draw[thick,->] (v2)--(v4);
    \end{tikzpicture}
    \caption{}\label{fig:ex_G}
  \end{subfigure}
  \begin{subfigure}[b]{.36\textwidth}
    \centering
    \begin{tikzpicture}[xscale=1.5,every node/.style={draw,circle,outer sep=1pt,inner
        sep=0pt,minimum size=1.5mm,fill=none}]
      \foreach \i in {1,...,4} { \node(v\i) at (-90*\i-135:1.2) [label={-90*\i-135:$v_\i$}]{}; }
      \foreach \i in {1,...,4} { \node(w\i) at (-90*\i-90:1.2) [label={-90*\i-90:$w_\i$}]{}; }
      \draw[thick,->] (w1)--(v2);
      \draw[thick,->] (w2)--(v3);
      \draw[thick,->] (w3)--(v2);
      \draw[thick,->] (w4)--(v3);
      \draw[thick,->] (w2)--(v1);
      \draw[thick,->] (w1)--(v3);
      \draw[thick,->] (w2)--(v4);
      \draw[thick,dashed,->] (v2)--(w2);
      \draw[thick,dashed,->] (v3)--(w3);
      \draw[thick,dashed,->] (v4)--(w4);
    \end{tikzpicture}
    \caption{}\label{fig:ex_G'}
  \end{subfigure}
  \caption{A digraph $G$, and the corresponding digraph $G'$ such that $\whBeta(G)$ is a GNF
    polytope for $G'$. The sets $E'_1$ and $E'_2$ are indicated by solid and dashed lines, respectively.}\label{fig:ex_G_G'}
\end{figure}
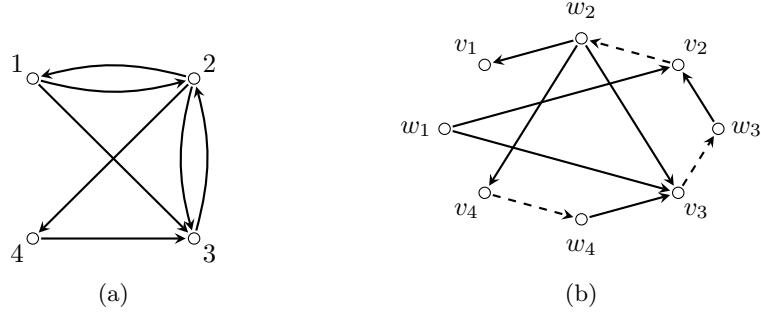
The constraints~\eqref{eq_wh:flow_extraction} to~\eqref{eq_wh:slacks} determine the supply/demand vector $\vect b$ as follows:
\begin{enumerate}
\item Constraint~\eqref{eq_wh:flow_injection} says that $w_1$ has a supply of $1$, that is, $b_{w_1}=1$.
\item Together with~\eqref{eq_wh:flow_extraction}, this implies that $v_1$ has a demand of
  $\beta^{n-1}$, that is, $b_{v_1}=-\beta^{n-1}$.
\item By constraint~\eqref{eq_wh:flow_bounds}, node $w_i$ for $2\leq i\leq n$ has a supply of
  $\beta^{n-1}$, that is $b_{w_i}=\beta^{n-1}$.
\item From~\eqref{eq_wh:flow_bounds} and~\eqref{eq_wh:flow_conservation}, it follows that, for $i=2,3,\dots,n$,
  \[y_i-\beta\sum_{j\in\ini(i)}x_{ji}=-\beta^{n-1},\] hence node $v_i$ has demand $\beta^{n-1}$,
  that is, $b_{v_i}=-\beta^{n-1}$. 
\end{enumerate}
The correspondence between the GNF in $G'$ and the model~\eqref{eq_wh:flow_extraction}
through~\eqref{eq_wh:nonnegativity} is illustrated in Figure~\ref{fig:GNF_conservation}.
To summarize, the vector $\vect b$ is given by
\[b_v=
  \begin{cases}
    1 & \text{for }v=w_1,\\
    \beta^{n-1}& \text{for }v\in\{w_2,w_3,\dots,w_n\},\\
    -\beta^{n-1}& \text{for }v\in\{v_1,v_2,\dots,v_n\}.
  \end{cases}
\]
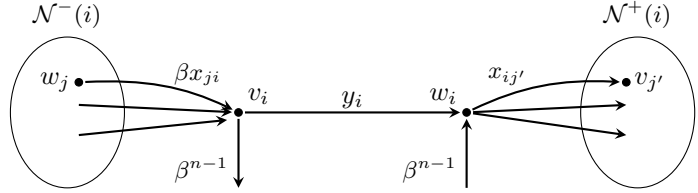
\begin{figure}
	\centering
	\begin{tikzpicture}[xscale=1.5, N/.style={ellipse,draw,fill=none,inner sep=0pt,minimum
            size=1.5mm}, base/.style={draw,fill=black,shape=circle,outer sep=1pt,inner sep=1pt,minimum size=.1cm}]
	\node[base] (v) at (0,0) [label={[label distance=-.1cm]80:$v_i$}] {};
	\node[base] (w) at (2,0) [label={[label distance=-.1cm]100:$w_i$}] {};
        \node[base] (w1) at (-1.4,.4) [label={[label distance=-.1cm]left:$w_j$}] {};
        \node[base] (v1) at (3.4,.4) [label={[label distance=-.1cm]right:$v_{j'}$}]{};
	
	\draw (-1.5,0) ellipse (.5cm and 1cm) node[draw=none,above=1cm,black]{\small $\ini(i)$};
	\draw (3.5,0) ellipse (.5cm and 1cm) node[draw=none,above=1cm,black]{\small $\outi(i)$};
	\draw [thick,->] (v) to node[draw=none,above=-2pt] {$y_i$} (w);
	\draw [thick,->] (-1.4,.1) to (v);
	\draw [thick,bend right=-20,->] (w1) to node[above,near end] {{\small $\beta x_{ji}$}} (v);
	\draw [thick,->] (-1.4,-.3) to (-.1,-.1);
	\draw [thick,->] (w) to (3.4,.1);
	\draw [thick,bend right=-20,->] (w) to node[above,near start] {{\small $x_{ij'}$}} (v1);
	\draw [thick,->] (w) to (3.4,-.3);
        \draw [thick,->] (v) to node[left,near end] {{\small $\beta^{n-1}$}} (0,-1);
        \draw [thick,->] (2,-1) to node[left,near start] {{\small $\beta^{n-1}$}} (w);
	\end{tikzpicture}
        \caption{Flow conservation for a node $i\in V\setminus\{1\}$. The two ovals represent the
        set $\{w_j\,:\,i\in\ini(i)\}$, and $\{v_j\,:\,i\in\outi(i)\}$, respectively. Flow
        conservation in $w_i$ corresponds to~\eqref{eq_wh:flow_bounds}, and flow conservation in
        $v_i$ is~\eqref{eq_wh:flow_conservation} (substituting $y_i+\beta^{n-1}$ for $\sum_{j\in\outi(j)}x_{ij}$).}\label{fig:GNF_conservation}
\end{figure}
The interpretation of $\whBeta(G)$ as a GNF polytope allows us to apply the known results about the
structure of bases for GNF polytopes to $\whBeta(G)$ and identify bases of $\whBeta(G)$ with
subgraphs of $G'$. For this purpose, we introduce some terminology (following \cite[Section
15.3]{ahuja1993network}), and then state the characterization of bases of GNF polytopes in
Theorem~\ref{thm:GNF_bases}. For a cycle $C$ (not necessarily directed) with a given orientation,
let $\overline C$ and $\underline C$ denote the sets of forward an backward arcs in $C$. The
\emph{cycle multiplier} is
\[\mu(C)=\frac{\prod_{(v,w)\in\overline C}\mu_{vw}}{\prod_{(v,w)\in\underline C}\mu_{vw}}.\]
If one unit of flow is sent along $C$, starting at some node $s$, then $\mu(C)$ units return to this
node. A cycle $C$ is called a \emph{breakeven cycle} if $\mu(C)=1$. 
\begin{definition}\label{def:augmented_tree}
  An \emph{augmented tree} is a connected graph with exactly one cycle, called the \emph{extra cycle}. An \emph{augmented forest}
  is a collection of node-disjoint augmented trees.
\end{definition}
\begin{definition}\label{def:good_augmented_tree}
  An augmented tree is called a \emph{good augmented tree} if its extra cycle is not a breakeven
  cycle. An augmented forest as a \emph{good augmented forest} if each of its components is a good
  augmented tree.
\end{definition}
In our specific setting it turns out that the breakeven condition is equivalent to having the same
number of forward and backward arcs.
\begin{definition}\label{def:balanced_cycle}
  An oriented cycle is called \emph{balanced} if it has the same number of forward and backward
  arcs.
\end{definition}
\begin{proposition}\label{prop:balanced_breakeven}
  A cycle $C$ in $G'$ is breakeven if and only if it is balanced.
\end{proposition}
\begin{proof}
  Let $C$ be a cycle in $G'$, and let $C_1$ and $C_2$ denote the set of arcs of $C$ in
  $E'_1$ and $E'_2$, respectively. As the multipliers of the arcs in $C_1$ are all equal to $\beta$
  and the multipliers of the arcs in $C_2$ are all equal to $1$, we have $\mu(C)=\beta^d$, where $d$
  is the difference between the number of forward arcs and the number of backward arcs in $C_1$. As
  a consequence, $C$ is a breakeven cycle if and only if $C_1$ contains the same number of forward
  and backward arcs. Let $e_1,e_2,\dots,e_k$ be the elements of $C_2$ in the order in which they are
  traversed by $C$, and let $P_1,\dots,P_k$ be the paths into which $C$ is cut by deleting the arcs
  in $C_2$. More precisely, $P_i$ is the path from $e_i$ to $e_{i+1}$ for $i=1,\dots,k-1$, and $P_k$
  is the path from $e_k$ to $e_1$. Let $d_i$ be the difference between the number of forward and
  backward arcs on $P_i$. If $P_i$ starts and ends at a forward arc then $d_i=1$, if it starts and
  ends at a backward arc, then $d_i=-1$, and otherwise $d_i=0$. Now we conclude, that $C$ is a
  breakeven cycle if and only if $C_1$ contains the same number of forward and backward arcs if and
  only if $d_1+\dots+d_k=0$ if and only if $C_2$ contains the same number of forward and backward
  arcs if and only if $C$ is balanced.
\end{proof}
From Proposition~\ref{prop:balanced_breakeven} we deduce that an augmented forest in $G'$ is good if
and only if none of its extra cycles is balanced. This is illustrated in
Figure~\ref{fig:augmented_tree_forest} where~(\subref{fig:augmented_tree})
and~(\subref{fig:augmented_forest}) depict an augmented tree and an augmented forest, respectively,
for the digraph $G'$ corresponding to the digraph $G$ in~(\subref{fig:graph_augmented}). The tree in
Figure~\ref{fig:augmented_tree_forest}(\subref{fig:augmented_tree}) is not a good augmented tree
because the extra cycle is balanced. The forest in
Figure~\ref{fig:augmented_tree_forest}(\subref{fig:augmented_forest}) is a good augmented forest as
the extra cycles are not balanced.
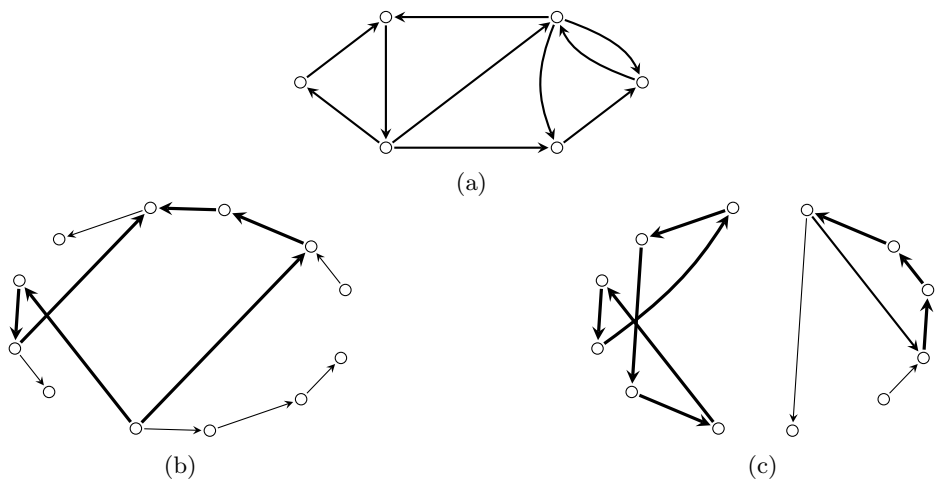
\begin{figure}
\centering
\begin{subfigure}[b]{\textwidth}
	\centering 
	\begin{tikzpicture}[xscale=2.25,every node/.style={draw,circle,outer sep=1pt,inner
            sep=0pt,minimum size=1.5mm,fill=none}]
	\foreach \i in{1,...,6} {
	\node(v\i) at (60*\i:1) {};
	} 
	
	\draw[->,thick] (v1) to (v2);
	\draw[->,thick] (v3) to (v2);
	\draw[->,thick] (v4) to (v3);
	\draw[->,thick] (v4) to (v5);
	\draw[->,thick] (v5) to (v6);
	\draw[->,thick, bend left=20] (v6) to (v1);
	\draw[->,thick, bend right=10] (v1) to (v5);
	\draw[->,thick] (v4) to (v1);
	\draw[->,thick] (v2) to (v4);
	\draw[->,thick, bend left=20] (v1) to (v6);
	\end{tikzpicture}
	\caption{}\label{fig:graph_augmented}
\end{subfigure}
\begin{subfigure}[b]{.45\textwidth}
	\centering 
	\begin{tikzpicture}[xscale=2.25,yscale=1.5,every node/.style={draw,circle,outer sep=1pt,inner
            sep=0pt,minimum size=1.5mm,fill=none}]          
          \foreach \i in{1,...,6} {
            \node(v\i) at (60*\i-20:1) {};
            \node(w\i) at (60*\i+15:1) {};
          }	
	\draw[->,very thick] (w1) to (v2);
	\draw[->,very thick] (w3) to (v2);
	\draw[->,very thick] (w4) to (v3);
	\draw[->] (w4) to (v5);
	\draw[->] (w6) to (v1);
	\draw[->,very thick] (w4) to (v1);
        \draw[->,very thick] (v1) to (w1);
        \draw[->] (v2) to (w2);
        \draw[->] (v5) to (w5);
        \draw[->,very thick] (v3) to (w3);
        \draw[->] (w3) -- (v4);
        \draw[->] (w5) -- (v6);
	\end{tikzpicture}
	\caption{}\label{fig:augmented_tree}
\end{subfigure}
\begin{subfigure}[b]{.45\textwidth}
	\centering 
	\begin{tikzpicture}[xscale=2.25,yscale=1.5,every node/.style={draw,circle,outer sep=1pt,inner sep=0pt,minimum size=1.5mm,fill=none}]
	\foreach \i in{1,...,6} {
            \node(v\i) at (60*\i-20:1) {};
            \node(w\i) at (60*\i+15:1) {};
          }	
	
	\draw[->,very thick,bend left=-10] (w3) to (v2);
	\draw[->,very thick] (w4) to (v3);
	\draw[->] (w5) to (v6);
        \draw[->] (w1) to (v5);
	\draw[->,very thick, bend left=0] (w6) to (v1);
	\draw[->,very thick] (w2) to (v4);
	\draw[->,thick, bend left=0] (w1) to (v6);
        \draw[->,very thick] (v2) to (w2);
        \draw[->,very thick] (v3) to (w3);
        \draw[->,very thick] (v4) to (w4);
        \draw[->,very thick] (v1) to (w1);
        \draw[->,very thick] (v6) to (w6);
	\end{tikzpicture}
	\caption{}\label{fig:augmented_forest}
\end{subfigure}
\caption{A digraph $G$ and two augmented forests in the corresponding digraph $G'$. The extra cycles
  are indicated by thick lines. The augmented tree (\subref{fig:augmented_tree}) is not good because
  its extra cycle is balanced. The augmented forest (\subref{fig:augmented_forest}) is good because
  none of the two extra cycles is balanced.}\label{fig:augmented_tree_forest}
\end{figure}

\begin{theorem}[Section 15.5 in~\cite{ahuja1993network}]\label{thm:GNF_bases}
  Let $X$ be a directed graph. For a partition $\arcset(X)=B\cup L\cup U$, the triple $(B,L,U)$ is a
  basis for a GNF polytope with underlying digraph $X$ if and only if $B$ is a good augmented forest
  which spans all the nodes of $X$.
\end{theorem}
 
We denote the GNF interpretation of the polytope $\whBeta(G)$ by $\nfBeta(G')$. In
Theorem~\ref{thm:augmented_tree_bases} we strengthen Theorem~\ref{thm:GNF_bases} for the
polytope $\nfBeta(G')$ by showing that in order to obtain a \emph{feasible} basis, the augmented
forest has to be connected, that is, it has to be a good augmented tree.
\begin{theorem}\label{thm:augmented_tree_bases}
  If $(B',L',U')$ is a feasible basis for $\nfBeta(G')$, then $B'$ is a good augmented tree.
\end{theorem}
This will be proved by showing that the assumption of more than one connected components leads to a
contradiction. For this purpose we need some preliminary results which are stated in
Lemmas~\ref{lem:nonempty_Vij}--\ref{lem:rsqz}. For the rest of this subsection, we assume that $(B',L',U')$ is a basis for
$\nfBeta(G')$. In particular, $B'$ is a good augmented forest, and we assume it has connected
components $T_1,\dots,T_m$. We shall use $T_k$ to denote both the arc set and the node set of the
good augmented tree $T_k$. The intended meaning will be clear from the context. Without loss of
generality, $w_1\in T_1$. For convenience, we set
\begin{align*}
  V_k &= \{i\in V\,:\,v_i\in T_k\text{ and }w_i\in T_k\},\\
  V^-_k &= \{i\in V\,:\,v_i\not\in T_k\text{ and }w_i\in T_k\},\\
  V^+_k &= \{i\in V\,:\,v_i\in T_k\text{ and }w_i\not\in T_k\}.
\end{align*}
Thus, $V_k$ correspond to the arcs in $E'_2$ which have both endpoints in $T_k$, $V^-_k$ corresponds
to the arcs in $E'_2$ which enter $T_k$, and $V^+_k$ corresponds to the arcs in $E'_2$ which leave
$T_k$. The next lemma states that every component $T_k$ must contain both $v_i$ and $w_i$ for some
$i\neq 1$.
\begin{lemma}\label{lem:nonempty_Vij}
  For every $k\in\{1,\dots,m\}$, the extra cycle in $T_k$ contains an arc from $E'_2$. In
  particular, $V_{1}\setminus\{1\}\neq\emptyset$ and $V_{k}\neq\emptyset$, for
  $k\in \{2,\dotsc,m\}$.
\end{lemma}
\begin{proof}
  Fix $k\in \{1,\dots,m\}$, let $C$ be the extra cycle in $T_k$, and assume that every arc in
  $C$ is in $E'_1$. Thus, in $C$ the nodes $w_i$ only have outgoing arcs and the nodes $v_i$
  only have incoming arcs. This implies that $C$ is balanced which is a contradiction.
\end{proof}
In the next lemma we analyze the interaction between components. The arcs in $E'_1$ can carry flow
only if they correspond to basic variables. As a consequence, the flow transfer between connected
components is only through arcs in $E'_2$, that is, arcs of the form $(v_i,w_i)$. 
\begin{lemma}\label{lem:total_incoming_flow} 
\begin{align}
  (1-\beta) \sum_{j\in V_{k}} \Phi(j) &= \beta\sum_{j \in V^-_k} \Phi(j) -
  \sum_{j\in V^+_k} \Phi(j)&&k=2,\dots,m,\label{eq:sum_of_flows_T2} \\
  (1-\beta) \sum_{j\in V_1} \Phi(j) &= 1-\beta^n+ \beta\sum_{j \in V^-_1} \Phi(j) -
                                         \sum_{j \in V^+_1} \Phi(j).\label{eq:sum_of_flows_T1}
\end{align}
\end{lemma}
\begin{proof}
  Let $k\in \{2,\dots,m\}$. Summing~\eqref{eq_wh:flow_conservation} over all $i\in V_{k}\cup
  V^-_k$ and then swapping the indices $i$ and $j$, and splitting the sum into the terms for $j\in
  V_{k}$ and the terms for $j\in V^-_k$, we obtain
  \begin{multline*}
    \sum_{i\in V_{k}\cup V^-_k} \sum_{j \in \outi(i)} x_{w_i\,v_j} =\beta\sum_{i\in V_k\cup
      V^-_k}\sum_{j\in \ini(i)} x_{w_j\,v_i}=\beta\sum_{j\in V_k\cup V^-_k}\sum_{i\in \ini(j)}
    x_{w_i\,v_j}\\
    =\beta\left(\sum_{j\in
      V_k}\sum_{i\in \ini(j)} x_{w_i\,v_j}+\sum_{j\in V^-_k}\sum_{i\in \ini(j)}
    x_{w_i\,v_j}\right)=\beta\left(\sum_{j\in V_k}\Phi(j)+\sum_{j\in V^-_k}\Phi(j)\right),
  \end{multline*}  
  All the nonzero terms on the left hand side have $j\in V_k\cup V^+_k$, so we can
  rearrange the left hand side as follows:
  \begin{multline*}
    \sum_{i\in V_k\cup V^-_k} \sum_{j \in \outi(i)} x_{w_i\,v_j} =\sum_{i\in V_k\cup V^-_k}
    \sum_{j \in \outi(i)\cap V_k} x_{w_i\,v_j}+\sum_{i\in V_k\cup V^-_k} \sum_{j \in
      \outi(i)\cap V^+_k} x_{w_i\,v_j} \\
    =\sum_{j\in V_k}\sum_{i\in\ini(j)}x_{w_i\,v_j}+\sum_{j\in V^+_k}\sum_{i\in\ini(j)}x_{w_i\,v_j}=\sum_{j\in V_k}\Phi(j)+\sum_{j\in V^+_k}\Phi(j),
  \end{multline*}
  where for the second equality we used that for $j\in V_k\cup V^+_k$ and $i\in\ini(j)$, the
  variable $x_{w_i\,v_j}$ can be nonzero only if $i\in V_k\cup V^-_k$. Therefore,
  \[\sum_{j\in V_k}\Phi(j)+\sum_{j\in V^+_k}\Phi(j)=\beta\left(\sum_{j\in
        V_k}\Phi(j)+\sum_{j\in V^-_k}\Phi(j)\right),\] which is equivalent
  to~\eqref{eq:sum_of_flows_T2}. For $k=1$, we proceed similarly. We start by summing
  constraints~\eqref{eq_wh:flow_extraction} and~\eqref{eq_wh:flow_conservation} over all
  $i\in V_1\cup V^-_1$: 
  \begin{multline*}
    \sum_{i\in V_{1}\cup V^-_1} \sum_{j \in \outi(i)}
    x_{w_i\,v_j} = 1- \beta^n + \beta\sum_{i\in V_1\cup V^-_1}\sum_{j \in \ini(i)} x_{w_j\,v_i}= 1-
    \beta^n + \beta\sum_{j\in V_1\cup V^-_1}\sum_{i \in \ini(j)} x_{w_i\,v_j}\\
    =1-
    \beta^n + \beta\left(\sum_{j\in V_1}\Phi(j)+\sum_{j\in V^-_1}\Phi(j)\right).
  \end{multline*}
  As above, the left hand side is equal to $\sum_{j\in V_1}\Phi(j)+\sum_{j\in V^+_1}\Phi(j)$,
  and therefore,
  \[\sum_{j\in V_1}\Phi(j)+\sum_{j\in V^+_1}\Phi(j)=1-\beta^n + \beta\left(\sum_{j\in
        V_1}\Phi(j)+\sum_{j\in V^-_1}\Phi(j)\right),\]
  which is equivalent to~\eqref{eq:sum_of_flows_T1}.
\end{proof}
Next we show that every component must have the same number of incoming and outgoing arcs in
$E'_2$.
\begin{lemma}\label{lem:equal_V_kl}
  For every $k\in\{1,\dots,m\}$, $\displaystyle\lvert V^-_k\rvert=\lvert V^+_k\rvert$.   
\end{lemma}
\begin{proof}
  For every $k\in\{1,\dots,m\}$, the left hand side of~\eqref{eq:sum_of_flows_T2}
  and~\eqref{eq:sum_of_flows_T1}, respectively, is
  \[(1-\beta) \sum_{j\in V_k} \Phi(j)=\delta\lvert V_k\rvert(1-O(\delta))=O(\delta).\]
  For $k=2,\dots,m$, the right hand side of~\eqref{eq:sum_of_flows_T2} is
  \[ \beta\sum_{j \in V^-_k} \Phi(j) - \sum_{j\in V^+_k} \Phi(j)=(1-\delta)\abs{V^-_k}\rvert(1-O(\delta))-\abs{V^+_k}(1-O(\delta))\\
    =\abs{V^-_k}-\abs{V^+_k}+O(\delta),\]  
  and the right side of~\eqref{eq:sum_of_flows_T1} is
  \begin{multline*}
    1-\beta^n+ \beta\sum_{j \in V^-_1} \Phi(j) - \sum_{j\in V^+_1}
    \Phi(j)=(1-\delta)\abs{V^-_1}(1-O(\delta))-\abs{V^+_1}(1-O(\delta))+O(\delta)\\
    =\abs{V^-_1}-\abs{V^+_1}+O(\delta).
  \end{multline*}
  In both cases, we conclude that $\abs{V^-_k}-\abs{V^+_k}=O(\delta)$, and the claim follows
  because the left hand side is an integer.
\end{proof}
For $i\in V^-_k\cup V^+_k$, we have $(v_i,w_i)\notin B'$. As a consequence, $i\in L'\cup U'$ which implies that 
$\Phi(i)\in\{1,\beta^{n-2}\}$. For $k\in\{1,\dots,m\}$, set
\begin{align*}
  r_k &= \abs{V^-_k}=\abs{V^+_k},&
  s_k &= \left\lvert\left\{i \in V^-_k\,:\,\Phi(i)=1\right\}\right\rvert,&\
  q_k&= \left\lvert\left\{i\in V^+_k\,:\,\Phi(i)=1\right\}\right\rvert.
\end{align*}
and for $k\in\{2,\dots,m\}$, $z_k = \left\lvert\left\{i\in
    V^+_{k}\,:\,\Phi(i)=\beta^{n-1}\right\}\right\rvert$, or equivalently, $z_k=1$ if $1\in V^+_{k}$,
and $z_k=0$ otherwise.
\begin{lemma}\label{lem:rsqz}
  \begin{align}
    (n-2)(s_k-q_k)+z_k-r_k &\geq 1\qquad k=2,\dots,m,\label{eq:thm_aug_tree_1}\\
    (n-2)(s_1-q_1)+n-r_1-(z_2+\cdots+z_m) &\geq 1.\label{eq:thm_aug_tree_2}
  \end{align}
\end{lemma}
\begin{proof}
  We start by bounding the left hand sides of the identities~\eqref{eq:sum_of_flows_T2}
  and~\eqref{eq:sum_of_flows_T1} from below:
  \begin{equation}\label{eq:LHS_asym_bound}
    \delta \sum_{j\in V_k} \Phi(j) \geq\delta\lvert V_k\rvert\beta^{n-1}=\delta\lvert
    V_k\rvert(1-O(\delta))\geq\delta+O(\delta^2),    
  \end{equation}
  where the second inequality follows from Lemma~\ref{lem:nonempty_Vij}. Next we express the sums on the
  right hand sides of~\eqref{eq:sum_of_flows_T2} and~\eqref{eq:sum_of_flows_T1} in terms of the
  quantities $r_k$, $s_k$, $q_k$ and $z_k$.  First consider the case $k\in\{2,\dots,m\}$. Then
  \begin{align*}
    \sum_{j \in V^-_k} \Phi(j) &= s_k + (r_k-s_k)\beta^{n-2}=r_k-(r_k-s_k)(n-2)\delta+O\left(\delta^2\right),\\
    \sum_{j \in V^+_k} \Phi(j) &= q_k + (r_k-q_k-z_k)\beta^{n-2}+z_k\beta^{n-1}=r_k-\left[(r_k-q_k)(n-2)+z_k\right]\delta+O\left(\delta^2\right).
  \end{align*}
  The right hand side of~\eqref{eq:sum_of_flows_T2} is
  \begin{multline*}
    (1-\delta)\sum_{j \in V^-_k} \Phi(j)-\sum_{j \in V^+_k} \Phi(j) \\
    = r_k-(r_k-s_k)(n-2)\delta-r_k\delta-r_k+\left[(r_k-q_k)(n-2)+z_k\right]\delta+O\left(\delta^2\right)\\
    =\left[(s_k-q_k)(n-2)-r_k+z_k\right]\delta+O\left(\delta^2\right),
  \end{multline*}
  and with~\eqref{eq:LHS_asym_bound} we obtain~\eqref{eq:thm_aug_tree_1}. For $k=1$,
  \begin{align*}
    \sum_{j \in V^-_1} \Phi(j) &= s_1 +
                                  (r_1-s_1-(z_2+\dots+z_m))\beta^{n-2}+(z_2+\dots+z_m)\beta^{n-1}\\
    &=r_1-\left[(r_1-s_1)(n-2)+(z_2+\dots+z_m)\right]\delta+O\left(\delta^2\right),\\
    \sum_{j \in V^+_1} \Phi(j) &= q_1 + (r_1-q_1)\beta^{n-2}=r_1-(r_1-q_1)(n-2)\delta+O\left(\delta^2\right).
  \end{align*}
  The right hand side of~\eqref{eq:sum_of_flows_T1} is
  \begin{multline*}
    1-\beta^n+ (1-\delta)\sum_{j \in V^-_1} \Phi(j) - \sum_{j \in V^+_1} \Phi(j) \\
    =n\delta+r_1-\left[(r_1-s_1)(n-2)+(z_2+\dots+z_m)\right]\delta-r_1\delta-r_1+(r_1-q_1)(n-2)\delta+O\left(\delta^2\right)\\
    =\left[n+(s_1-q_1)(n-2)-r_1-(z_2+\dots+z_m)\right]\delta+O\left(\delta^2\right).
  \end{multline*}
  and with~\eqref{eq:LHS_asym_bound} we obtain~\eqref{eq:thm_aug_tree_2}.
\end{proof}
We are now ready to finish the proof of Theorem~\ref{thm:augmented_tree_bases}.
\begin{proof}[Proof of Theorem~\ref{thm:augmented_tree_bases}]  
  It follows from~\eqref{eq:thm_aug_tree_1} that $r_k\geq 1$, for $k=2,\dots,m$, because $r_k=0$
  implies $s_k=q_k=z_k=0$ which forces the left hand side of~\eqref{eq:thm_aug_tree_1} to be
  zero. Then~\eqref{eq:thm_aug_tree_1} implies $(n-2)(s_k-q_k)\geq 1+r_k-z_k\geq 1$, hence
  $s_k\geq q_k+1$. On the other hand, by definition we have
  \[s_1+\dots+s_m=q_1+\cdots+q_m=\left\lvert\left\{i\in V\setminus \bigcup_{k=1}^{m}V_k\,:\, \Phi(i)=1\right\}
    \right\rvert.\]
  Therefore, $q_1-s_1=(s_2+\cdots+s_m)-(q_2+\cdots+q_m)\geq m-1\geq 1$, and~\eqref{eq:thm_aug_tree_2} implies
  \[n-2\leq(n-2)(q_1-s_1)\leq n-r_1-(z_2+\cdots+z_m)-1,\]
  which simplifies to $r_1+z_2+\dots+z_m\leq 1$. If $r_1=0$, then~\eqref{eq:sum_of_flows_T1} becomes
  \[\sum_{j\in V_1}\Phi(j)=\frac{1-\beta^n}{1-\beta}=1+\beta+\beta^2+\cdots+\beta^{n-1}=n+O\left(\delta\right).\]
  But $\Phi(j)=1-O(\delta)$ for every $j\in V$, so the left hand side is
  $\abs{V_1}+O(\delta)$, and this implies $\abs{V_1}=n$ which contradicts the assumption that
  $B'$ is not connected. Therefore, $r_1=1$ and $z_2+\dots+z_m=0$,~\eqref{eq:thm_aug_tree_2} becomes
  $(n-2)(s_1-q_1+1)\geq 0$, and together with $s_1-q_1\leq 1-m$, we obtain $m=2$ and
  $q_1=s_1+1$. Furthermore, $q_1\leq r_1$, this implies $q_1=1$ and $s_1=0$, and then $\Phi(j)=1$
  for the unique node $j\in V^+_{1}$, and $\Phi(j')=\beta^{n-2}$ for the unique node $j'\in
  V^-_{1}$. Now~\eqref{eq:sum_of_flows_T1} becomes
  \[(1-\beta)\sum_{j\in V_1}\Phi(j)=1-\beta^n+\beta^{n-1}-1=(1-\beta)\beta^{n-1}.\] From
  $z_2+\cdots+z_m=0$, we have $1\in V_1$, and with $\Phi(1)=\beta^{n-1}$, this implies
  $\sum_{j\in V_1\setminus\{1\}} \Phi(j)=0$. But then $V_1=\{1\}$, which contradicts
  Lemma~\ref{lem:nonempty_Vij}.
\end{proof}

\subsection{Thick and thin arcs}\label{subsec:thick_and_thin}
In this subsection, we show that for a feasible basis of the polytope $\whBeta(G)$, the
values of the basic variables either tend to one or zero as $\beta\to 1$. More precisely, we show
that there are exactly $n$ arcs corresponding to basic variables that tend to 1, and that these $n$
arcs form a collection of node-disjoint cycles which we call \emph{thick cycles}. Furthermore, we
provide upper bounds for the number of non-basic $y$-variables which are at their lower bounds and
the number of non-basic $y$-variables which are at their upper bounds.

\begin{definition}\label{def:thick_and_thin}
  Let $(B,L,U)\in \Gamma(G)$, where $B=A\cup Y$ and let $\vect{x}= A_B^{-1}\vect{b}$. The arc $(i,j)\in A$ is called \emph{thick}
  with respect to $B$, if $x_{ij} =1-O(\delta)$. The arc $(i,j)\in A$ is called \emph{thin} with
  respect to $B$ if $x_{ij}=O(\delta)$. Otherwise, the arc $(i,j)\in A$ is called \emph{intermediate}
  with respect to $B$.
\end{definition}
We start by expressing the sum of all flow variables in terms of $\delta$.
\begin{lemma}\label{lem:total_outflow}
  For every point $\vect x\in\whBeta(G)$,
  \[\sum_{(i,j)\in E}x_{ij}=\sum_{i=1}^n\Phi(i)=\sum_{i=1}^n\Psi(i)=\binom{n}{1}-\binom{n}{2}\delta+\binom{n}{3}\delta^2-\dots+(-1)^{n-1}\binom{n}{n}\delta^{n-1}.\]
\end{lemma}
\begin{proof}
  The first two equalities follow from the observation that $\sum_{i=1}^n\Phi(i)$ and
  $\sum_{i=1}^n\Psi(i)$ are two different ways of computing the sum of all the $x_{ij}$, first by
  grouping the arcs according to their end node, and second by grouping them according to their
  start node. With $\Psi(1)=1$, $\Phi(1)=\beta^{n-1}$ and $\Psi(i)=\beta\Phi(i)$ for all
  $i\in V\setminus\{1\}$, we obtain
  \[\sum_{i=1}^n\Psi(i)-1=\sum_{i=2}^n\Psi(i)=\beta\sum_{i=2}^n\Phi(i)=\beta\left(\sum_{i=1}^n\Phi(i)-\beta^{n-1}\right)=\beta\left(\sum_{i=1}^n\Psi(i)-\beta^{n-1}\right),\]
  which implies
  \begin{multline*}
    \sum_{i=1}^n\Psi(i)=\frac{1-\beta^n}{1-\beta}=\sum_{k=0}^{n-1}\beta^k=\sum_{k=0}^{n-1}(1-\delta)^k=\sum_{k=0}^{n-1}\sum_{l=0}^k\binom{k}{l}(-\delta)^l\\
    =\sum_{l=0}^{n-1}\left((-1)^l\sum_{k=l}^{n-1}\binom{k}{l}\right)\delta^l=\sum_{l=0}^{n-1}(-1)^l\binom{n}{l+1}\delta^l.  \qedhere
  \end{multline*}  
\end{proof}
\begin{theorem}\label{thm:disjoint_cycle_structure}
  Let $(B,L,U)\in \Gamma(G)$, where $B=A\cup Y$ and let $A_1$, $A_2$ denote the sets of thick and thin arcs, respectively. Then
  \begin{enumerate}[(i)]
  \item $A =  A_1 \cup A_2$,\label{item:no_intermediate_arcs}
  \item $A_1$ forms a spanning collection of node-disjoint directed cycles,\label{item:thick_cycles}
  \item for every $i\in \nodeset \setminus \{1\}$, the digraph with arc set $A_1\cup A_2$
    contains a directed path from node $1$ to node $i$,\label{lem:directed_path_from_1}
   \item $\lvert L\rvert\leq (n-1)/2$ and $\lvert U\rvert\leq (n-1)/2$.\label{item:bounds_LU}
  \end{enumerate}
\end{theorem}
\begin{proof}\hfill
  \begin{enumerate}[(i)]
  \item For the sake of contradiction, assume that there exists an intermediate arc $(i,j) \in
    A$. It follows from~\eqref{eq_wh:flow_extraction}--\eqref{eq_wh:flow_bounds} that, for every
    $i\in \nodeset$, the total outflow and the total inflow tend to one as $\beta\to 1$. Thus, there
    is at least one more intermediate arc leaving node $i$, say $(i,\ell)$. Similarly, nodes $j$ and
    $\ell$ have another intermediate incoming arc. This argument shows, that every node with an
    outgoing intermediate arc has at least two outgoing intermediate arcs, and every node with an
    incoming intermediate arc has at least two incoming intermediate arcs. As a consequence $A$
    contains a cycle $C$ which alternates between forward and backward arcs, which implies that the
    corresponding cycle $C'$ in $B'$ is balanced. This contradicts the assumption that $B'$ is a
    basis for $\nfBeta(G')$.
  \item Using~(\ref{eq_wh:flow_bounds}),~(\ref{eq_wh:slacks}), and~(\ref{eq_wh:flow_injection}), for every $i\in \nodeset$,
    \[\left\lvert\{j\in\outi(i)\,:\,(i,j)\in A_1\}\right\rvert - O(\delta)=
      \sum_{j\in\outi(i)\,:\,(i,j)\in A_1}x_{ij}\leq\sum_{j\in\outi(i)}x_{ij}=1 - O(\delta).\]
    Thus, there is at most one thick arc leaving node $i$. If there is a node $i\in \nodeset$ with no
    thick leaving arc, then using part~(\ref{item:no_intermediate_arcs}), the lower bound from the
    wedge constraints~\eqref{wedge_constraints} and~(\ref{eq_wh:flow_extraction}), we have
    \[ 1 - O(\delta) \leq \sum_{j \in \outi(i)} x_{ij} = \sum_{j \in \outi(i)\,:\, (i,j) \in A_2}
      x_{ij} = O(\delta),\] which is a contradiction. Hence, in the subgraph $X=(\nodeset,A_1)$
    every node has exactly one leaving arc. A similar argument shows that every node has exactly one
    entering thick arc, and therefore, $A_1$ is a spanning collection of node-disjoint directed
    cycles.
  \item For the sake of contradiction, let us assume that there is a node $i^*\in \nodeset$ such
    that there is no directed path from node 1 to $i^*$. Let $V^*\subseteq\nodeset$ be the set of
    nodes $j$ such that there exists a directed path from $j$ to $i^*$. Since $1\notin V^*$,~\eqref{eq_wh:flow_conservation} implies
    \[ \sum_{\ell \in \outi(j)} x_{j\ell} = \beta \sum_{\ell \in \ini(j)} x_{\ell j},\]
     for all
    $j\in V^*$. Summing over all $j\in V^*$, we obtain
    \[\sum_{j\in V^*} \sum_{\ell \in \outi(j)} x_{j\ell} = \beta \sum_{j\in V^*} \sum_{\ell \in
        \ini(j)} x_{\ell j}.\]
    The sum on the left hand side is over all arcs starting in $V^*$ and
    we split it into the sum over the arcs with both nodes in $V^*$ and the arcs starting in $V^*$
    and ending in $\nodeset\setminus V^*$. The sum on the right hand side is over all arcs ending in
    $V^*$, and by the definition of $V^*$ these arcs also start in $V^*$. This gives
    \[\sum_{\substack{(j,\ell)\in \arcset\\\ell\in V\setminus V^*,\, j\in V^*}} x_{j\ell} + 
      \sum_{\substack{(j,\ell)\in \arcset\\ \ell\in V^*,\, j\in V^*}} x_{j\ell} = \beta\sum_{\substack{(\ell,j)\in \arcset\\\ell\in V^*,\, j\in V^*}} x_{\ell j}= \beta\sum_{\substack{(j,\ell)\in \arcset\\\ell\in V^*,\, j\in V^*}} x_{j\ell},\]
    which simplifies to
    \[ \sum_{\substack{(j,\ell)\in \arcset\\\ell\in V\setminus V^*,\, j\in V^*}} x_{j\ell}=
      -\delta\sum_{\substack{(j,\ell)\in \arcset\\\ell\in V^*,\, j\in V^*}} x_{j\ell}.\]
    With $\delta>0$ and the non-negativity of the $x_{j\ell}$, it follows that 
    \[\sum_{\substack{(j,\ell)\in \arcset\\\ell\in V\setminus V^*,\, j\in V^*}}
      x_{j\ell}=\sum_{\substack{(j,\ell)\in \arcset\\\ell\in V^*,\, j\in V^*}} x_{j\ell}=0,\]
    and this implies that $x_{j\ell}=0$ for all $j\in V^*$, $\ell\in\outi(j)$. In particular,
    \[\sum_{\ell\in\outi(i^*)}x_{i^*\ell}=0<\beta^{n-1},\]
    which contradicts the wedge constraint~\eqref{wedge_constraints}.
  \item We use Lemma~\ref{lem:total_outflow}, together with $\Psi(i)=\beta=1-\delta$ for $i\in U$,
    and $\Psi(i)=\beta^{n-1}=1-(n-1)\delta+O\left(\delta^2\right)$ for $i\in L$:
    \begin{multline*}
       \sum_{i\in Y}\Psi(i)=\sum_{i=1}^n\Psi(i)-\Psi(1)-\lvert U\rvert(1-\delta)-\lvert
       L\rvert(1-(n-1)\delta)+O\left(\delta^2\right)\\
       =n-\binom{n}{2}\delta-1-\abs{U}-\abs{L}+\abs{U}\delta+(n-1)\abs{L}\delta+O\left(\delta^2\right)\\
       =\abs{Y}-\left(\binom{n}{2}-\abs{U}-(n-1)\abs{L}\right)\delta+O\left(\delta^2\right).
     \end{multline*}
     On the other hand, from $\abs{Y}=n-1-\abs{U}-\abs{L}$ and $\beta^{n-1}\leq \Psi(i)\leq
   \beta$ for every $i\in Y$, we obtain
     \[\left(n-1-\abs{U}-\abs{L}\right)(1-(n-1)\delta)-O\left(\delta^2\right)\leq\sum_{i\in Y}\Psi(i)\leq \left(n-1-\abs{U}-\abs{L}\right)(1-\delta),\]
     and therefore,
\[n-1-\abs{U}-\abs{L}\leq\binom{n}{2}-\abs{U}-(n-1)\abs{L}\leq (n-1)\left(n-1-\abs{U}-\abs{L}\right)\]
The first of these inequalities simplifies to $\abs{L}\leq (n-1)/2$, and the second one to
$\abs{U}\leq(n-1)/2$.\qedhere
  \end{enumerate}
\end{proof}
Combining Theorems~\ref{thm:augmented_tree_bases} and~\ref{thm:disjoint_cycle_structure} we
establish the following results on the structure of feasible bases of $\whBeta(G)$ in the GNF
setting.
\begin{theorem}\label{thm:bases_structure}
  Let $(B,L,U)\in\Gamma(G)$ where $B=Y\cup A_1\cup A_2$ and let $B'=Y'\cup A_1'\cup A'_2$ be the
  corresponding arc sets in $G'$, that is,
  \begin{align*}
    Y' &= \{(v_i,w_i)\,:\,i\in Y\},& A'_1 &= \{(w_i,v_j)\,:\,(i,j)\in A_1\}, & A'_2 &= \{(w_i,v_j)\,:\,(i,j)\in A_2\}.
  \end{align*}
  \begin{enumerate}[(i)]
  \item $B'$ is a good augmented tree spanning all nodes of $G'$.\label{item:good_augtree}
  \item $A'_1$ is the unique perfect matching in $G'$.\label{item:matching}
  \item $Y'\cup A'_1$ is a collection of $n-\lvert Y\rvert$ node-disjoint
    paths.\label{item:disjoint_paths}
  \end{enumerate}
\end{theorem}
\begin{proof}
  Item~(\ref{item:good_augtree}) follows from
  Theorem~\ref{thm:augmented_tree_bases}. Theorem~\ref{thm:disjoint_cycle_structure}
  (\ref{item:thick_cycles}) implies that the arcs in $A'_1$ form a perfect matching in $G'$. With
  Lemma~\ref{lem:nonempty_Vij}, there is at least one arc of the form $(v_i,w_i)$ in the extra cycle
  of $B'$. As a consequence, $B'- Y'$ is a forest. Since a forest has at most one perfect matching,
  $A'_1$ is the unique perfect matching. The last statement can be verified by adding the arcs in
  $Y'$ one-by-one to $A'_1$. Clearly, we start with $n$ node-disjoint paths, each of them being a
  single arc. And then in each step the new arc from $Y'$ connects the endpoints of two paths,
  thereby reducing the number of paths by 1.
\end{proof}
The relation between the characterizations of feasible bases for the polytope $\whBeta(G)$ provided
in Theorems~\ref{thm:disjoint_cycle_structure} and~\ref{thm:bases_structure} is illustrated in Figure~\ref{fig:ex_bases}.
\begin{figure}
  \centering
  \begin{subfigure}[b]{\textwidth}          
          \centering 
	\begin{tikzpicture}[xscale=2.5,every node/.style={draw,circle,outer sep=1pt,inner
            sep=0pt,minimum size=1.5mm,fill=none}]
	\foreach \i in{1,2,4,5,6} {
	\node(v\i) at (60*\i:1) {};
      }
      \node[fill] (v3) at (180:1) {};
	\draw[->,thick] (v1) to (v2);
	\draw[->,thick] (v3) to (v2);
	\draw[->,thick] (v4) to (v3);
	\draw[->,thick] (v4) to (v5);
	\draw[->,thick] (v5) to (v6);
	\draw[->,thick, bend left=20] (v6) to (v1);
	\draw[->,thick, bend right=10] (v1) to (v5);
	\draw[->,thick] (v4) to (v1);
	\draw[->,thick] (v2) to (v4);
	\draw[->,thick, bend left=20] (v1) to (v6);
      \end{tikzpicture}
      \caption{An example graph where the black filling indicates node $1$.}\label{fig:graph_augmented_2}
    \end{subfigure}
    \begin{subfigure}[b]{.45\textwidth}
      \centering
      \begin{tikzpicture}[xscale=2.5,yscale=1.5,every node/.style={draw,circle,outer sep=1pt,inner
          sep=0pt,minimum size=1.5mm,fill=none}]		
        \node[fill=black] (v1) at (180:1) {};
        \useasboundingbox (-1.05,-1.2) rectangle (1.55,1.2);
                \node (v2) at (120:1) [label={[label distance=-.4cm]90:{\small $\beta-\beta^5$}}] {};
                \node[fill=gray] (v3) at (60:1) [label={[label distance=0cm]90:{\small $0$}}] {};
                \node (v4) at (0:1) [label={[label distance=.05cm]0:{\small $\beta^4-\beta^5$}}] {};
                \node (v5) at (-60:1) [label={[label distance=-.4cm]-90:{\small $\beta^3-\beta^5$}}] {};
                \node (v6) at (-120:1) [label={[label distance=-.4cm]-90:{\small $\beta^2-\beta^5$}}] {};
		\draw[->,very thick] (v1) to node[draw=none,left,inner sep=.1cm] {{\small $1$}} (v2);
		\draw[->,very thick] (v2) to node[draw=none,right,inner sep=.05cm] {{\small $\beta$}} (v6);
		\draw[->,very thick] (v6) to node[draw=none,below,inner sep=.1cm,near end] {{\small $\beta^{5}$}} (v1);
		\draw[->,very thick] (v3) to node[draw=none,left,inner sep=.05cm] {{\small
                    $\beta^{5}$}} (v5);
		\draw[->,very thick] (v5) to node[draw=none,below,inner sep=.1cm,near end] {{\small
                    $\beta^{3}$}} (v4);
		\draw[->,very thick] (v4) to node[draw=none,below,inner sep=.1cm] {{\small $\beta^{4}$}} (v3);
		\draw[->] (v6) to node[draw=none,above,rectangle] {{\small $\beta^2-\beta^{5}$}} (v5);
		\draw[->,bend left=45] (v3) to node[draw=none,above,inner sep=.1cm] {{\small
                    $0$}} (v4);                
		\end{tikzpicture}
		\caption{A feasible basis as a subgraph of $G$.}\label{fig:ex_bases_G}
	\end{subfigure}
	\begin{subfigure}[b]{.45\textwidth}
		\centering 
		\begin{tikzpicture}[scale=1,every node/.style={draw,circle,outer sep=1pt,inner
			sep=0pt,minimum size=1.5mm,fill=none}] 
		
                      \node[fill=black] (w1) at (-.5,0.35) {};
                      \node[fill=black] (v1) at (-.5,-.35) {};
                      \node(v2) at (1,1) {};
                      \node(w2) at (1.7,1) {};
                      \node(v6) at (1.7,-.8) {};
                      \node(w6) at (1.7,-1.5) {};		
                      \node(v5) at (2.7,-1.5) {};
                      \node(w5) at (3.4,-1.5) {};
                      \node(v4) at (5,-0.35) {};
                      \node(w4) at (5,.35) {};
                      \node[fill=gray] (v3) at (3.4,1) {};
                      \node[fill=gray] (w3) at (2.7,1) {};
		
		\draw[->,very thick] (w1) to (v2);
		\draw[->,very thick] (w2) to (v6);
		\draw[->,very thick] (w6) to (v1);
		\draw[->,very thick] (w3) to (v5);
		\draw[->,very thick] (w5) to (v4);
		\draw[->,very thick] (w4) to (v3);
		\draw[->] (w6) to (v5);
		\draw[->] (w3) to (v4);
		\draw[->,thick, dashed] (v2) to (w2);
		\draw[->,thick, dashed] (v4) to (w4);
		\draw[->,thick, dashed] (v5) to (w5);
		\draw[->,thick, dashed] (v6) to (w6);
                \useasboundingbox (-.6,-2) rectangle (5.1,1.1);
		\end{tikzpicture}
		\caption{The same basis as a subgraph of $G'$.}\label{fig:ex_bases_G'}
	\end{subfigure}
	\caption{A feasible basis $(B,L,U)$ for the graph in~(\subref{fig:graph_augmented_2}) is
          illustrated in~(\subref{fig:ex_bases_G}). Here $B$ contains the $x$-variables for the
          shown arcs and the $y$-variables for the empty nodes, while $L$ is the set with the grey
          node as its only element, and $U$ is the empty set. The basic feasible solution for this
          basis is indicated by the arc and node labels, and the thin and thick arcs are indicated
          by the strength of the lines. The good augmented tree in $G'$ corresponding to this basis
          is shown in~(\subref{fig:ex_bases_G'}). In particular, the thick arcs form a spanning
          collection of node-disjoint cycles in~(\subref{fig:ex_bases_G}), and a perfect matching in~(\subref{fig:ex_bases_G'}).}\label{fig:ex_bases}
\end{figure}
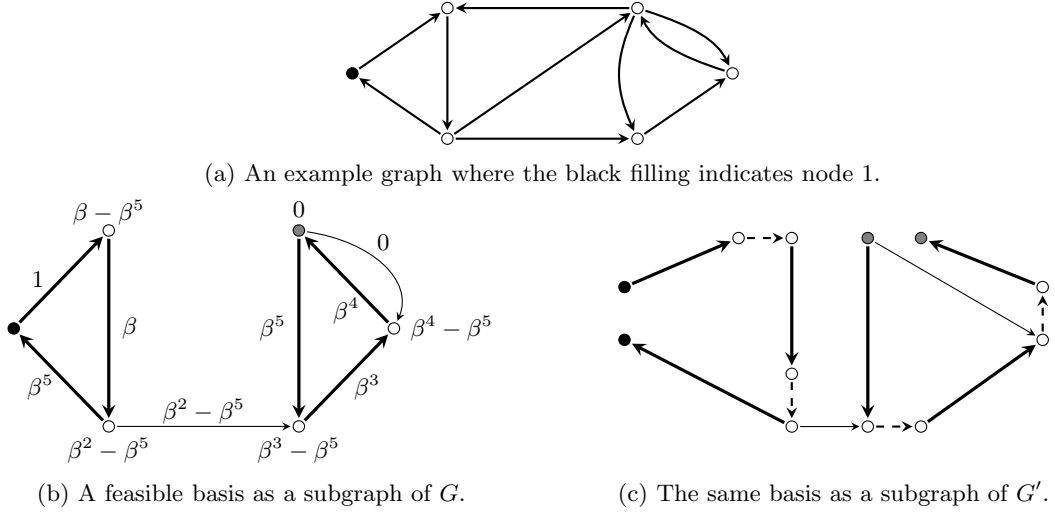

%%%Local Variables:
%%% mode: latex
%%% TeX-master: "paper"
%%% End:

%% file: 4_class_of_bases.tex
\section{A class of feasible bases}\label{sec:class_of_bases}
For the approach to the Hamiltonian Cycle Problem based on sampling extreme points of $\whBeta(G)$, an
essential ingredient is a lower bound for the fraction of Hamiltonian extreme points in the set of
all extreme points. In order to derive such an estimate it will be useful to understand the
structure of the feasible bases corresponding to a fixed subgraph of $G$ or $G'$. By
Theorem~\ref{thm:disjoint_cycle_structure}, a subgraph of $G$ which corresponds to a feasible basis
contains a spanning collection of node-disjoint directed cycles. In $G'$, this corresponds to a
perfect matching between the sets $\{w_i\,:\,i\in V\}$ and $\{v_i\,:\,i\in V\}$, and
Theorem~\ref{thm:bases_structure} provides the additional information that the subgraph
corresponding to the basic variables is a good augmented spanning tree.

In this section we make a further step by focusing on the feasible bases corresponding to a fixed
(isomorphism class of) good augmented tree in $G'$. More precisely, we characterize the feasible
bases whose corresponding good augmented tree is a Hamilton cycle in $G'$. This does not imply
that the basis is Hamiltonian in the sense that it corresponds to a Hamilton cycle in $G$ (see
Figure~\ref{fig:special_basis_1} for an illustration). From Section~\ref{subsec:gnf} we know that at
least one $y$-variable needs to be in the basis, that is, the good augmented tree must contain at
least one arc of the form $(v_i,w_i)$. We restrict our attention further to the case that the basis
contains exactly one $y$-variable, say $y_s$. Starting with the thick arcs $(w_i,v_j)$, and adding
the arc $(v_s,w_s)$ corresponding to the basic $y$-variable, we obtain $n-2$ isolated arcs and one
path of length $3$ which consists of the thick arc into $v_s$, the arc $(v_s,w_s)$ and the thick arc
leaving $w_s$. In order to obtain a Hamilton cycle, the thin arcs have to connect these
components in such a way that for every $i\in V\setminus\{s\}$, the node $w_i$ has exactly one
outgoing arc, and node $v_i$ has exactly one incoming arc. In $G$, this corresponds to every node
$i\in V\setminus\{s\}$ having exactly one thick and one thin outgoing arc, and exactly one thick and
one thin incoming arc, while the only arcs incident with node $s$ are the thick arcs into and out of
$s$. Up to relabeling the variables, we can then assume that the thick arc into node $s$ is the arc
$(s-1,s)$, and that the thick arc out of node $s$ is the $s$-th thick arc if we traverse the
Hamilton cycle in $G'$ starting with the thick arc out of $w_1$. More precisely, up to relabeling
the variables, we may assume that there is a fixed-point free permutation $\pi$ of $V$ such that the
following conditions are satisfied:

\vspace{-\topsep}
\begin{itemize}
\item The set of thick arcs is $\{(i,\pi(i))\,:\,i\in V\}$.
\item The set of thin arcs is $\left\{(i+1,\pi(i))\,:\,i\in[n]\setminus\{s\}\right\}$.
\item $\pi(i)=i+1$ if and only if $i=s-1$.
\end{itemize}\vspace{-\topsep}

In the last condition above, and throughout this section, whenever $i\in V$, $i\pm 1$ refers to the
node $i\pm 1\pmod n$. For a basis $(B,L,U)$, the set $B$ of basic variables is determined by the
permutation $\pi$, and we still have to choose a partition $V\setminus\{1,s\}=L\cup U$ of the set of
non-basic $y$-variables such that $y_i=0$ for $i\in L$ and
$y_i=\beta-\beta^{n-1}=(n-2)\delta-\binom{n-1}{2}\delta^2+O\left(\delta^3\right)$ for $i\in U$. In Figure~\ref{fig:special_bases}, we illustrate how a basis can be completely specified by a
picture, using different types of nodes to indicate if $i\in L$, $i\in U$ or $y_i$ is a basis variable. The general structure of the bases considered in this section is illustrated in Figure~\ref{fig:HC_one_yvar}.

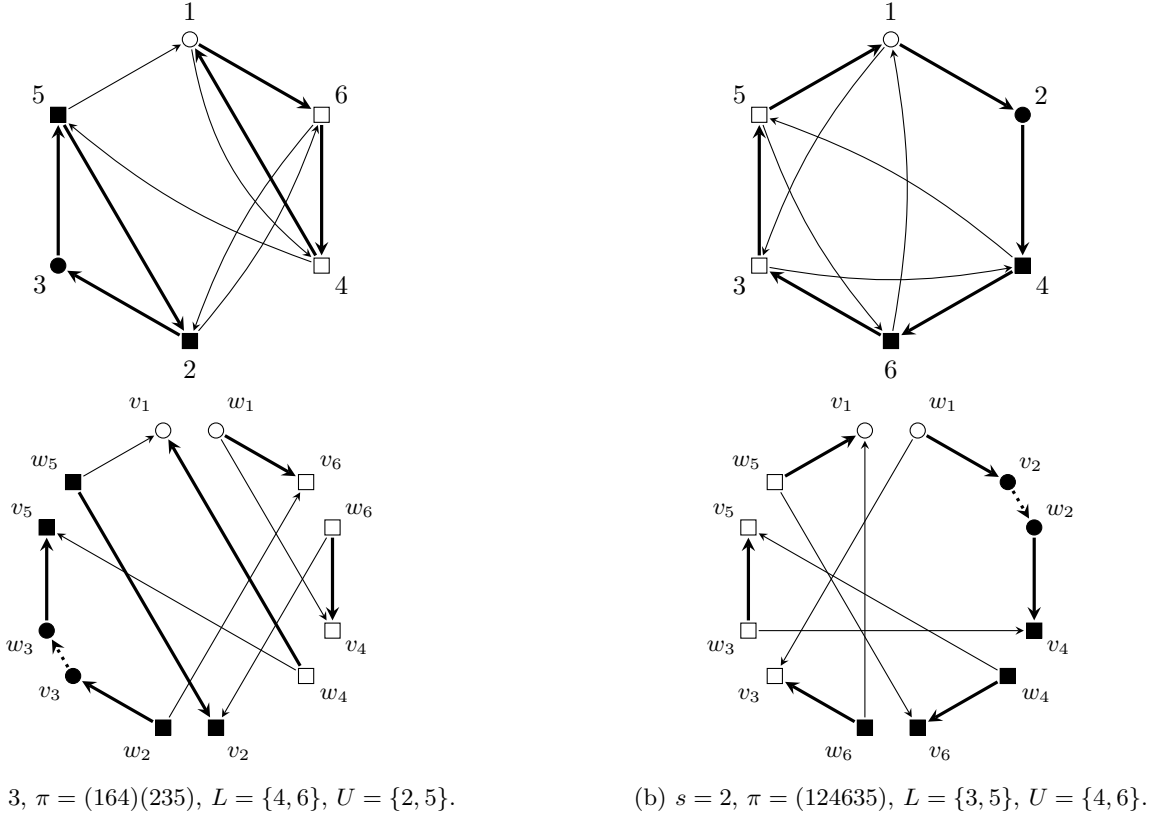
\begin{figure}[htb]
  \centering
  \begin{subfigure}{.45\linewidth}
    \centering
    \begin{tikzpicture}[scale=2,
      U/.style={rectangle,draw,fill=black,inner sep=0pt,outer sep=1pt,minimum size=2mm},
      L/.style={rectangle,draw,fill=none,inner sep=0pt,outer sep=1pt,minimum size=2mm},
      base/.style={circle,draw,fill=black,inner sep=0pt,outer sep=1pt,minimum size=2mm},
      start/.style={circle,draw,fill=none,inner sep=0pt,outer sep=1pt,minimum size=2mm}]
    \node[start,label={[label distance=-0cm]90:$1$}] (v1) at (90:1) {};
    \node[L,label={[label distance=-.1cm]30:$6$}] (v2) at (30:1) {};
    \node[L,label={[label distance=-.1cm]-30:$4$}] (v3) at (-30:1) {};
    \node[U,label={[label distance=-0cm]-90:$2$}] (v4) at (-90:1) {};
    \node[base,label={[label distance=-.1cm]-150:$3$}] (v5) at (-150:1) {};
    \node[U,label={[label distance=-.1cm]-210:$5$}] (v6) at (-210:1) {};    
    \draw[very thick,->] (v1) -- (v2);
    \draw[very thick,->] (v2) -- (v3);
    \draw[very thick,->] (v3) -- (v1);
    \draw[very thick,->] (v4) -- (v5);
    \draw[very thick,->] (v5) -- (v6);
    \draw[very thick,->] (v6) -- (v4);    
    \draw[bend right=20,->] (v1) to (v3);
    \draw[bend right=10,->] (v2) to (v4);
    \draw[bend right=-10,->] (v3) to (v6);
    \draw[bend right=10,->] (v4) to (v2);
    \draw[bend right=0,->] (v6) to (v1);    
    \end{tikzpicture} \hfill 
    \centering
    \begin{tikzpicture}[scale=2,
      U/.style={rectangle,draw,fill=black,inner sep=0pt,outer sep=1pt,minimum size=2mm},
      L/.style={rectangle,draw,fill=none,inner sep=0pt,outer sep=1pt,minimum size=2mm},
      base/.style={circle,draw,fill=black,inner sep=0pt,outer sep=1pt,minimum size=2mm},
      start/.style={circle,draw,fill=none,inner sep=0pt,outer sep=1pt,minimum size=2mm}]
      \node[start,label={[label distance=-0cm]100:{\small $v_1$}}] (v1) at (100:1) {};
      \node[start,label={[label distance=-0cm]80:{\small $w_1$}}] (w1) at (80:1) {};
      \node[L,label={[label distance=-.1cm]40:{\small $v_6$}}] (v2) at (40:1) {};
      \node[L,label={[label distance=-.1cm]20:{\small $w_6$}}] (w2) at (20:1) {};
      \node[L,label={[label distance=-.1cm]-20:{\small $v_4$}}] (v3) at (-20:1) {};
      \node[L,label={[label distance=-.1cm]-40:{\small $w_4$}}] (w3) at (-40:1) {};
      \node[U,label={[label distance=-0cm]-80:{\small $v_2$}}] (v4) at (-80:1) {};
      \node[U,label={[label distance=-0cm]-100:{\small $w_2$}}] (w4) at (-100:1) {};
      \node[base,label={[label distance=-.1cm]-140:{\small $v_3$}}] (v5) at (-140:1) {};
      \node[base,label={[label distance=-.1cm]-160:{\small $w_3$}}] (w5) at (-160:1) {};
      \node[U,label={[label distance=-.1cm]-200:{\small $v_5$}}] (v6) at (-200:1) {};
      \node[U,label={[label distance=-.1cm]-220:{\small $w_5$}}] (w6) at (-220:1) {};    
    \draw[very thick,->] (w1) -- (v2);
    \draw[very thick,->] (w2) -- (v3);
    \draw[very thick,->] (w3) -- (v1);
    \draw[very thick,->] (w4) -- (v5);
    \draw[very thick,->] (w5) -- (v6);
    \draw[very thick,->] (w6) -- (v4);    
    \draw[bend right=0,->] (w1) to (v3);
    \draw[bend right=0,->] (w2) to (v4);
    \draw[bend right=0,->] (w3) to (v6);
    \draw[bend right=0,->] (w4) to (v2);
    \draw[bend right=0,->] (w6) to (v1);
    \draw[very thick,dotted,bend right=0,->] (v5) to (w5);    
    \end{tikzpicture} \caption{$s=3$, $\pi=(164)(235)$, $L=\{4,6\}$, $U=\{2,5\}$.}\label{fig:special_basis_1}
  \end{subfigure}\hfill
\centering
\begin{subfigure}{.45\linewidth}
	\centering
	\begin{tikzpicture}[scale=2,
	U/.style={rectangle,draw,fill=black,inner sep=0pt,outer sep=1pt,minimum size=2mm},
	L/.style={rectangle,draw,fill=none,inner sep=0pt,outer sep=1pt,minimum size=2mm},
	base/.style={circle,draw,fill=black,inner sep=0pt,outer sep=1pt,minimum size=2mm},
	start/.style={circle,draw,fill=none,inner sep=0pt,outer sep=1pt,minimum size=2mm}]
	\node[start,label={[label distance=-0cm]90:$1$}] (v1) at (90:1) {};
	\node[base,label={[label distance=-.1cm]30:$2$}] (v2) at (30:1) {};
	\node[U,label={[label distance=-.1cm]-30:$4$}] (v3) at (-30:1) {};
	\node[U,label={[label distance=-0cm]-90:$6$}] (v4) at (-90:1) {};
	\node[L,label={[label distance=-.1cm]-150:$3$}] (v5) at (-150:1) {};
	\node[L,label={[label distance=-.1cm]-210:$5$}] (v6) at (-210:1) {};    
	\draw[very thick,->] (v1) -- (v2);
	\draw[very thick,->] (v2) -- (v3);
	\draw[very thick,->] (v3) -- (v4);
	\draw[very thick,->] (v4) -- (v5);
	\draw[very thick,->] (v5) -- (v6);
	\draw[very thick,->] (v6) -- (v1);    
	\draw[bend right=10,->] (v5) to (v3);
	\draw[bend right=10,->] (v3) to (v6);
	\draw[bend right=10,->] (v6) to (v4);
	\draw[bend right=10,->] (v4) to (v1);
	\draw[bend right=10,->] (v1) to (v5);    
	\end{tikzpicture} \hfill
	\centering
	\begin{tikzpicture}[scale=2,
	U/.style={rectangle,draw,fill=black,inner sep=0pt,outer sep=1pt,minimum size=2mm},
	L/.style={rectangle,draw,fill=none,inner sep=0pt,outer sep=1pt,minimum size=2mm},
	base/.style={circle,draw,fill=black,inner sep=0pt,outer sep=1pt,minimum size=2mm},
	start/.style={circle,draw,fill=none,inner sep=0pt,outer sep=1pt,minimum size=2mm}]
	\node[start,label={[label distance=-0cm]100:{\small $v_1$}}] (v1) at (100:1) {};
	\node[start,label={[label distance=-0cm]80:{\small $w_1$}}] (w1) at (80:1) {};
	\node[base,label={[label distance=-.1cm]40:{\small $v_2$}}] (v2) at (40:1) {};
	\node[base,label={[label distance=-.1cm]20:{\small $w_2$}}] (w2) at (20:1) {};
	\node[U,label={[label distance=-.1cm]-20:{\small $v_4$}}] (v3) at (-20:1) {};
	\node[U,label={[label distance=-.1cm]-40:{\small $w_4$}}] (w3) at (-40:1) {};
	\node[U,label={[label distance=-0cm]-80:{\small $v_6$}}] (v4) at (-80:1) {};
	\node[U,label={[label distance=-0cm]-100:{\small $w_6$}}] (w4) at (-100:1) {};
	\node[L,label={[label distance=-.1cm]-140:{\small $v_3$}}] (v5) at (-140:1) {};
	\node[L,label={[label distance=-.1cm]-160:{\small $w_3$}}] (w5) at (-160:1) {};
	\node[L,label={[label distance=-.1cm]-200:{\small $v_5$}}] (v6) at (-200:1) {};
	\node[L,label={[label distance=-.1cm]-220:{\small $w_5$}}] (w6) at (-220:1) {};    
	\draw[very thick,->] (w1) -- (v2);
	\draw[very thick,->] (w2) -- (v3);
	\draw[very thick,->] (w3) -- (v4);
	\draw[very thick,->] (w4) -- (v5);
	\draw[very thick,->] (w5) -- (v6);
	\draw[very thick,->] (w6) -- (v1);    
	\draw[bend right=0,->] (w5) to (v3);
	\draw[bend right=0,->] (w3) to (v6);
	\draw[bend right=0,->] (w4) to (v1);
	\draw[bend right=0,->] (w1) to (v5);
	\draw[bend right=0,->] (w6) to (v4);
	\draw[very thick,dotted,bend right=0,->] (v2) to (w2);    
      \end{tikzpicture} \caption{$s=2$, $\pi=(124635)$, $L=\{3,5\}$, $U=\{4,6\}$.}\label{fig:special_basis_2}
\end{subfigure}
\caption{A quasi-Hamiltonian feasible basis (right) and a non-quasi-Hamiltonian (left) feasible
  basis for $n=6$ vertices. In the top pictures the bases are shown as subgraphs of $G$, and node
  types indicate the sets $L$ (empty squares), $U$ (filled squares), basic $y$-variable (filled
  circle), and node $1$ (empty circle). The pictures on the bottom illustrate that the corresponding
  good augmented tree is indeed a Hamilton cycle in $G'$.}\label{fig:special_bases}
\end{figure}

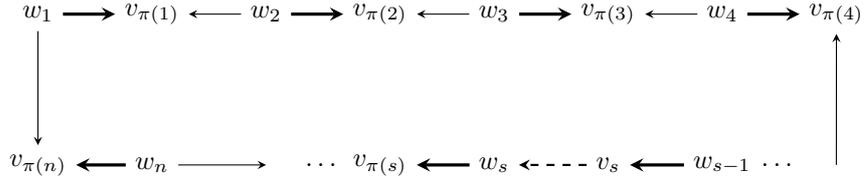
\begin{figure}[htb]
  \centering
  \begin{tikzpicture}[xscale=1.5]
	\node (w1) at (0,0) {$w_{1}$};
	\node (v1) at (1,0) {$v_{\pi(1)}$};
	
	\node (w2) at (2,0) {$w_{2}$};
	\node (v2) at (3,0) {$v_{\pi(2)}$};
   
	\node (w3) at (4,0) {$w_{3}$};
	\node (v3) at (5,0) {$v_{\pi(3)}$};
	
	\node (w4) at (6,0) {$w_{4}$};
	\node (v4) at (7,0) {$v_{\pi(4)}$};
	
	\node (wq) at (6,-2) {$w_{s-1}$};
	\node (vs) at (5,-2) {$v_{s}$};
	\node (ws) at (4,-2) {$w_{s}$};
	\node (vq) at (3,-2) {$v_{\pi(s)}$};
	
	\node (wk) at (1,-2) {$w_{n}$};
	\node (vk) at (0,-2) {$v_{\pi(n)}$};
	
	\draw[very thick,->] (w1) -- (v1);
	\draw[very thick,->] (w2) -- (v2);
	\draw[very thick,->] (w3) -- (v3);
	\draw[very thick,->] (w4) -- (v4);
	\draw[very thick,->] (wq) -- (vs);
	\draw[thick,dashed,->] (vs) -- (ws);
	\draw[very thick,->] (ws) -- (vq);
	\draw[very thick,->] (wk) -- (vk);

	\draw[->] (w2) -- (v1);
	\draw[->] (w3) -- (v2);
	\draw[->] (w4) -- (v3);
	\draw[->] (7,-2) -- (v4);
	\draw[->] (w1) -- (vk);
	\draw[->] (wk) -- (2,-2);
	\node at (2.5,-2) {\dots};
	\node at (6.5,-2) {\dots};
  \end{tikzpicture}
  \caption{The structure of the bases considered in this section. The dashed arc indicates the $y$-variable.}
  \label{fig:HC_one_yvar}
\end{figure}

Our aim is to characterize the 4-tuples $(s,\pi,L,U)$ of a node $s\in\{2,3,\dots,n\}$, a fixed-point
free permutation $\pi$ with $\pi(i)=i+1$ if and only if $i=s-1$, and a partition
$[n]=\{1,s\}\cup L\cup U$, such that the basis given by the triple $(B,L,U)$ is feasible,
where $B=\{s\}\cup\{(i,\pi(i))\,:\,i\in[n]\}\cup\{(i,\pi(i-1))\,:\,i\in[n]\setminus\{s\}\}$.
The cycles of the permutation $\pi$ correspond to the thick cycles in $B$. In particular,
the basis $(B,L,U)$ is quasi-Hamiltonian if and only if the permutation $\pi$ is a single cycle.

In order to simplify the notation, we introduce new variables. Let $\xi_i=x_{i\,\pi(i)}$, $i\in[n]$
be the flow on the thick arc leaving node $i$, and let $\eta_i=x_{i+1\,\pi(i)}$,
$i\in[n]\setminus\{s-1\}$ be the flow on the thin arc entering node
$\pi(i)$. This is illustrated in Figure~\ref{fig:xi_eta}.
For all $i\in[n]\setminus\{s\}$, $\xi_i+\eta_{i-1}=\Psi(i)$, and
$\xi_{\pi^{-1}(i)}+\eta_{\pi^{-1}(i)}=\Phi(i)$, and substituting the values of $\Phi(i)$ and
$\Psi(i)$ leads to the following representation of the
system~\eqref{eq_wh:flow_extraction}--\eqref{eq_wh:flow_bounds}:
\begin{align}
  \xi_s-\beta\xi_{s-1} &= 0,\label{eq:simplified_constraints_0}\\
  \xi_{\pi^{-1}(1)}+\eta_{\pi^{-1}(1)} &=\beta^{n-1}, \label{eq:simplified_constraints_1a}\\
  \xi_{\pi^{-1}(i)}+\eta_{\pi^{-1}(i)} &=\beta^{n-2} && i\in
                                                        L,\label{eq:simplified_constraints_1b}\\
  \xi_{\pi^{-1}(i)}+\eta_{\pi^{-1}(i)} &=1 && i\in U,\label{eq:simplified_constraints_1c}\\
  \xi_1+\eta_n &= 1,\label{eq:simplified_constraints_2a}\\
  \xi_i+\eta_{i-1} &=\beta^{n-1} && i\in L,\label{eq:simplified_constraints_2b}\\
  \xi_i+\eta_{i-1} &=\beta && i\in U,\label{eq:simplified_constraints_2c}\\
  \xi_s-y_s&=\beta^{n-1}.\label{eq:simplified_constraints_3}
\end{align}

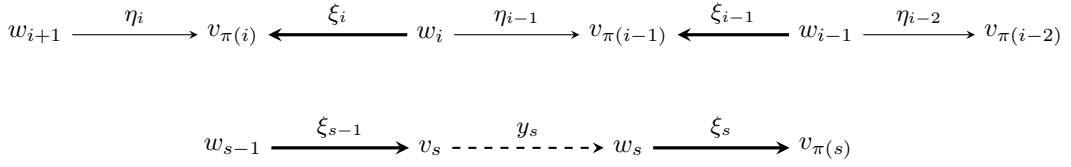
\begin{figure}[htb]
	\centering
	\begin{tikzpicture}[xscale=1.3]
	\node (w1) at (0,0) {$w_{i+1}$};
	\node (v1) at (2,0) {$v_{\pi(i)}$};
	\node (w2) at (4,0) {$w_{i}$};
	\node (v2) at (6,0) {$v_{\pi(i-1)}$};
	\node (w3) at (8,0) {$w_{i-1}$};
	\node (v3) at (10,0) {$v_{\pi(i-2)}$};
	\draw[->] (w1) to node[above] {{\small $\eta_{i}$}} (v1);
	\draw[->] (w2) to node[above] {{\small $\eta_{i-1}$}} (v2);
	\draw[very thick,->] (w2) to node[above] {{\small $\xi_{i}$}} (v1);
	\draw[->] (w3) to node[above] {{\small $\eta_{i-2}$}} (v3);
	\draw[very thick,->] (w3) to node[above] {{\small $\xi_{i-1}$}} (v2);
	\node (ws) at (6,-1.5) {$w_{s}$};
	\node (vs) at (4,-1.5) {$v_{s}$};
	\node (vs1) at (8,-1.5) {$v_{\pi(s)}$};
	\node (ws1) at (2,-1.5) {$w_{s-1}$};
	\draw[very thick,->] (ws) to node[above] {{\small $\xi_{s}$}} (vs1);
	\draw[very thick,->] (ws1) to node[above] {{\small $\xi_{s-1}$}} (vs);
	\draw[thick,dashed,->] (vs) to node[above] {{\small $y_{s}$}} (ws);
	\end{tikzpicture}
	\caption{The variables $\xi_i$ and $\eta_i$}
	\label{fig:xi_eta}
\end{figure}

\begin{lemma}\label{lem:characterisation_eta}
  The quadruple $(s,\pi,L,U)$ encodes a feasible basis if and only if the unique solution of
  system~(\ref{eq:simplified_constraints_0})--(\ref{eq:simplified_constraints_3}) satisfies
  $\eta_i\geq 0$ for all $i\in[n]\setminus\{s-1\}$, and $\beta^{n-1}\leq\xi_s\leq\beta$.
\end{lemma}
\begin{proof}
  If the quadruple $(s,\pi,L,U)$ encodes a feasible basis, then $\eta_i\geq 0$ for all
  $i\in[n]\setminus\{s-1\}$, and $0\leq y_s\leq\beta-\beta^{n-1}$, or equivalently,
  $\beta^{n-1}\leq\xi_s\leq\beta$. For the converse, let $(\vect{\xi},\vect{\eta})$ be the solution
  of the system~\eqref{eq:simplified_constraints_0}--\eqref{eq:simplified_constraints_3} for
  $(s,\pi,L,U)$, and assume that $\eta_i\geq 0$ for all $i\in[n]\setminus\{s-1\}$, and
  $\beta^{n-1}\leq\xi_s\leq\beta$. To verify feasibility, we only need to show that $\xi_i\geq 0$
  for all $i\in[n]$. For $i=s$, this is an immediate consequence of the assumption
  $\beta^{n-1}\leq\xi_s\leq\beta$ which implies $\xi_s=1+O(\delta)$. Using that the right hand sides
  of equations~\eqref{eq:simplified_constraints_1a} to~\eqref{eq:simplified_constraints_2c} are all
  $1+O(\delta)$, we obtain, for every $i\in[n]\setminus\{s\}$
  \[\xi_i=-\eta_{i-1}+1+O(\delta)=-\left(-\xi_{i-1}+1+O(\delta)\right)+1+O(\delta)=\xi_{i-1}+O(\delta),\]
  and by induction $\xi_i=1+O(\delta)>0$ for all $i\in [n]$.
\end{proof}
Lemma~\eqref{lem:characterisation_eta} implies that in order to characterize the feasible bases, we
need to derive necessary and sufficient conditions for $\eta_i \geq 0$ for all
$i\in [n]\setminus\{s-1\}$ and $\beta^{n-1}\leq\xi_s\leq\beta$. We start by finding necessary
conditions. Expressing $\xi_s$ in terms of $\abs{L}$ and $\abs{U}$ (Lemma~\ref{lem:xi_s}) allows us
to specify $\abs{L}$ and $\abs{U}$ (Lemma~\ref{lem:cardinality_LU}), and deduce that for every
feasible basis $\pi(s)\in U$ (Lemma~\ref{lem:nodes_of_long_path}). As a consequence, we find
asymptotic expressions for $\eta_s$ (Lemma~\ref{lem:value_eta_s}). Eliminating the $\xi$-variables
from the system~\eqref{eq:simplified_constraints_0}--\eqref{eq:simplified_constraints_3} leads to a
recursion for the $\eta_i$, where the base case is given by the value $\eta_s$ determined in
Lemma~\ref{lem:value_eta_s}. The asymptotic expression obtained from solving this recursion is given
in Lemma~\ref{lem:asymptotic_eta}. Assuming that $L$ and $U$ have the right cardinalities and
$\pi(s)\in U$, the second order term of $\eta_i$ is positive whenever the first order term vanishes
(Lemma~\ref{lem:second_order}). Theorem~\ref{thm:char_1} combines the necessary conditions obtained
so far into a characterization of feasible $(s,\pi,L,U)$ by three properties: (i) the cardinality
condition for $L$ and $U$, (ii) $\pi(s)\in U$, and (iii) the non-negativity of the linear terms in
the asymptotic expansions of the variables $\eta_i$.

\begin{lemma}\label{lem:xi_s}
  Let $(\vect\xi,\vect\eta)$ be the solution of the
  system~(\ref{eq:simplified_constraints_0})--(\ref{eq:simplified_constraints_3}) for
  $(s,\pi,L,U)$. Then
  \[\xi_s=1-\left[\binom{n}{2}-\abs{U}-(n-1)\abs{L}\right]\delta+\left[\binom{n}{3}-\binom{n-1}{2}\abs{L}\right]\delta^2+O\left(\delta^3\right).\]  
\end{lemma}
\begin{proof}
 We have
  \begin{multline*}
    n-\binom{n}{2}\delta+\binom{n}{3}\delta^2+O\left(\delta^3\right)=\sum_{i=1}^n\Psi(i)=1+\xi_s+\abs{L}\beta^{n-1}+\abs{U}\beta\\
    =1+\abs{L}+\abs{U}-\left(\abs{U}+(n-1)\abs{L}\right)\delta+\abs{L}\binom{n-1}{2}\delta^2+\xi_s+O\left(\delta^3\right)\\
    =n-1-\left(\abs{U}+(n-1)\abs{L}\right)\delta+\abs{L}\binom{n-1}{2}\delta^2+\xi_s+O\left(\delta^3\right),
    \end{multline*}
    where the first equation comes from Lemma~\ref{lem:total_outflow}, the second one follows from
    $\Psi(1)=1$, $\Psi(s)=\xi_s$, $\Psi(i)=\beta^{n-1}$ for $i\in L$ and $\Psi(i)=\beta$ for $i\in
    U$, the third one from substituting $1-\delta$ for $\beta$, and the last one from
    $[n]=\{1,s\}\cup L\cup U$. The claim follows by solving for $\xi_s$. 
  \end{proof}
 % Next we show that the cardinalities of $U$ and $L$ can differ by at most $1$.
\begin{lemma}\label{lem:cardinality_LU}
  If $(s,\pi,L,U)$ encodes a feasible basis, then $\abs{U} = \lfloor (n-2)/2\rfloor$ and
  $\abs{L} = \lceil (n-2)/2\rceil$. 
\end{lemma}
\begin{proof}
  We start with $\abs{L}+\abs{U}=n-2$. If $n$ is even, then
  Theorem~\ref{thm:disjoint_cycle_structure}(\ref{item:bounds_LU}) implies $\abs{L}\leq(n-2)/2$ and
  $\abs{U}\leq(n-2)/2$, and the claim follows. For odd $n$,
  Theorem~\ref{thm:disjoint_cycle_structure}(\ref{item:bounds_LU}) implies only
  $\abs{L},\abs{U}\leq (n-1)/2$, and therefore $\{\abs{L},\abs{U}\}=\{(n-1)/2,(n-3)/2\}$. For the
  sake of contradiction, assume $\abs{U}=(n-1)/2$ and $\abs{L}=(n-3)/2$. By Lemma~\ref{lem:xi_s},
 \begin{multline*}   \xi_s=1-\left[\binom{n}{2}-\frac{n-1}{2}-\frac{(n-1)(n-3)}{2}\right]\delta+\left[\binom{n}{3}-\binom{n-1}{2}\frac{n-3}{2}\right]\delta^2+O(\delta^3)\\
 =1-(n-1)\delta-\frac12\binom{n-1}{3}\delta^2+O\left(\delta^3\right)<\beta^{n-1},
\end{multline*}
which is the required contradiction.
\end{proof}
\begin{lemma}\label{lem:nodes_of_long_path}
	If $(s,\pi,L,U)$ encodes a feasible basis then $\pi(s)\in U$. 
\end{lemma}
\begin{proof}
	Suppose $\pi(s)\in L\cup \{1\}$. From Lemmas~\ref{lem:xi_s} and \ref{lem:cardinality_LU}, we have
	that $\xi_s=1-(n/2)\delta+O\left(\delta^2\right)$ if $n$ is even, and
	$\xi_s=1-\delta+O\left(\delta^2\right)$ if $n$ is odd. From~\eqref{eq:simplified_constraints_1a}
	and~\eqref{eq:simplified_constraints_1b}, it follows that
	\[ \eta_s\leq\beta^{n-2}-\xi_s\leq-(n-2)\delta-1+(n/2)\delta+O\left(\delta^2\right) =(2-n/2)\delta+O\left(\delta^2\right)<0,\]  
	which is the required contradiction.  
\end{proof}
In the next lemma, we use Lemma~\ref{lem:xi_s} to determine the asymptotics for the variables
$\xi_s$ and $\eta_s$ assuming the necessary conditions from Lemma~\ref{lem:cardinality_LU} and~\ref{lem:nodes_of_long_path}.
\begin{lemma}\label{lem:value_eta_s}
  Suppose that $\abs{L}=\lceil(n-2)/2\rceil$ and $\abs{U}=\lfloor(n-2)/2\rfloor$ and $\pi(s)\in U$.
  \begin{enumerate}[(i)]
  \item If $n$ is even then
    $\displaystyle\eta_s=1-\xi_s=(n/2)\delta + \frac{(n-1)(n-2)(n-6)}{12}\delta^2 +
    O\left(\delta^3\right)$.
  \item If $n$ is odd then
    $\displaystyle\eta_s=1-\xi_s=\delta + \frac{(n-1)(n-2)(n-3)}{12}\delta^2 +O\left(\delta^3\right)$.
  \end{enumerate}
\end{lemma}
\begin{proof}
  We apply Lemma~\ref{lem:xi_s}. For even $n$,
  \begin{multline*}
    \xi_s=1-\left[\binom{n}{2}-\frac{n-2}{2}-\frac{(n-1)(n-2)}{2}\right]\delta+\left[\binom{n}{3}-\binom{n-1}{2}\frac{n-2}{2}\right]\delta^2+O\left(\delta^3\right)\\
    = 1-(n/2)\delta-\frac{(n-1)(n-2)(n-6)}{12}\delta^2+O\left(\delta^3\right),
  \end{multline*}
  and for odd $n$,
  \begin{multline*}
    \xi_s=1-\left[\binom{n}{2}-\frac{n-3}{2}-\frac{(n-1)^2}{2}\right]\delta+\left[\binom{n}{3}-\binom{n-1}{2}\frac{n-1}{2}\right]\delta^2+O\left(\delta^3\right)\\
    = 1-\delta-\frac{(n-1)(n-2)(n-3)}{12}\delta^2+O\left(\delta^3\right).
  \end{multline*}
  In both cases, the claim follows since $\eta_s=1-\xi_s$ by~\eqref{eq:simplified_constraints_1c}
  for $i=\pi(s)\in U$.
\end{proof}

We set $R=\{1\}$, and partition the set $[n]\setminus\{s,s-1\}$ into nine sets $W_{PQ}$ with
$P,Q\in\{L,U,R\}$, where $W_{PQ}=\{i\in[n]\setminus\{s\}\,:\,i\in P,\,\pi(i)\in Q\}$. In particular,
$W_{RR}=\emptyset$. For every $i\in[n]\setminus\{s,s-1\}$, we have $\eta_i+\xi_i = \Phi(\pi(i))$ and
$\eta_{i-1}+\xi_i=\Psi(i)$. Eliminating the $\xi$-variables, we obtain the recursion
\begin{align}
\eta_i= \eta_{i-1} + \gamma_i&&i\in [n]\setminus \{s-1,s\},\label{eq:recursive_formula}
\end{align}
where,
   \begin{align}
   \gamma_i=\Phi(\pi(i))-\Psi(i)=
   \begin{cases}
   1-\beta & \text{if $i\in W_{UU}$},\\
   \beta^{n-2}-\beta & \text{if $i\in W_{UL}$},\\
   1-\beta^{n-1} & \text{if $i\in W_{LU}$},\\
   \beta^{n-2}-\beta^{n-1} & \text{if $i\in W_{LL}$},\\
   1 & \text{if $i\in W_{RU}$},\\
   \beta^{n-2} & \text{if $i\in W_{RL}$},\\
   \beta^{n-1}-\beta & \text{if $i\in W_{UR}$},\\
   0&\text{if $i\in W_{LR}$}.
   \end{cases}\label{gamma_values}
   \end{align}
   We substitute $1-\delta$ for $\beta$ and take the limit for $\delta\to 0$ to obtain
   $\gamma_i=\alpha_{i1}\delta+\alpha_{i2}\delta^2+O\left(\delta^3\right)$, where
  \begin{align}
    \alpha_{i1} &=
                  \begin{cases}
                    1 & \text{if }i\in W_{UU}\cup W_{LL},\\
                    3-n& \text{if }i\in W_{UL},\\
                    n-1 & \text{if }i\in W_{LU},\\
                    2-n& \text{if }i\in W_{UR}\cup W_{RL},\\
                    0& \text{if }i\in W_{LR}\cup W_{RU}.
                  \end{cases}& \alpha_{i2} &=
                                             \begin{cases}
                                               \frac 1 2 (n-2)(n-3)& \text{if }i\in W_{UL},\\
                                               -\frac 1 2 (n-1)(n-2) & \text{if }i\in W_{LU},\\
                                               \frac 1 2 (n-1)(n-2)& \text{if }i\in W_{UR},\\
                                               \frac 1 2 (n-2)(n-3)& \text{if }i\in W_{RL},\\
                                               2-n& \text{if }i\in W_{LL},\\
                                               0 & \text{if }i\in W_{UU}\cup W_{LR}\cup W_{RL}.
                                             \end{cases}\label{alpha_values}
  \end{align}
  The base of the recursion is given by the values for $\eta_s$ in
  Lemma~\ref{lem:value_eta_s}, and in order to capture this, we set $\alpha_{s1}= n/2$,
  $\alpha_{s2}= \frac 1 {12}(n-1)(n-2)(n-6)$ if $n$ is even and $\alpha_{s1}=1$,
  $\alpha_{s2}= \frac 1 {12}(n-1)(n-2)(n-3)$ if $n$ is odd. For $i\in[n]\setminus\{s,s-1\}$, we define
  the set $I(i)$ by
\[I(i)=\begin{cases}
\{s+1,\dots,i\} & \text{if } s+1\leq i\leq n\\
\{s+1,\dots,n,1,2,\dots,i\} & \text{if } 1\leq i\leq s-2,
\end{cases}\]
and put $N_{PQ}(i)=\abs{I(i)\cap W_{PQ}}$ for all $(P,Q)\in\{L,U,R\}^2$. The following lemma is
obtained from Lemma~\ref{lem:value_eta_s} and repeated application of~\eqref{eq:recursive_formula}.
\begin{lemma}\label{lem:asymptotic_eta}
  If $\abs{L}=\lceil(n-2)/2\rceil$, $\abs{U}=\lfloor(n-2)/2\rfloor$ and $\pi(s)\in U$, then for all $i\in[n]\setminus\{s,s-1\}$,
  \[\eta_i= \left(\alpha_{s1}+ \sum_{j\in I(i)}\alpha_{j1}\right)\delta + \left(\alpha_{s2}+\sum_{j\in I(i)}\alpha_{j2}
    \right)\delta^2+O\left(\delta^3\right).\]
\end{lemma}
We want to derive necessary and sufficient conditions for $\eta_i\geq 0$. Clearly, it is necessary
that $\alpha_{s1}+\sum_{j\in I(i)}\alpha_{j1}\geq 0$, and if the inequality is strict then this is also
sufficient. The next lemma deals with the equality case.
\begin{lemma}\label{lem:second_order}
  Let $n\geq 5$ and suppose $\abs{L}=\lceil(n-2)/2\rceil$, $\abs{U}=\lfloor(n-2)/2\rfloor$ and $\pi(s)\in U$. For all $i\in[n]\setminus\{s,s-1\}$,
  \[\alpha_{s1}+\sum_{j\in I(i)}\alpha_{j1}= 0\implies\alpha_{s2}+\sum_{j\in I(i)}\alpha_{j2}>0.\]
\end{lemma}
\begin{proof}
  By assumption
  \[\alpha_{s1}+N_{UU}(i)+N_{LL}(i)-(n-3)N_{UL}(i)+(n-1)N_{LU}(i)-(n-2)N_{UR}(i)-(n-2)N_{RL}(i)=0,\]
  and thus
  \begin{equation}\label{eq:assumption_even}
    (n-3)N_{UL}(i)-(n-1)N_{LU}(i)=\alpha_{s1}+N_{UU}(i)+N_{LL}(i)-(n-2)N_{UR}(i)-(n-2)N_{RL}(i).
  \end{equation}
  For the second order term, we obtain
  \begin{multline*}
    \frac{2}{n-2}\left(\alpha_{s2}+\sum_{j\in I(i)}\alpha_{j2}\right)=\frac{2\alpha_{s2}}{n-2}+(n-3)N_{UL}(i)-(n-1)N_{LU}(i)\\
    +(n-1)N_{UR}(i)+(n-3)N_{RL}(i)-2N_{LL}(i)\\
    \stackrel{\eqref{eq:assumption_even}}{=}
    \frac{2\alpha_{s2}}{n-2}+\alpha_{s1}+N_{UU}(i)+N_{UR}(i)-N_{RL}(i)-N_{LL}(i)\\
    \geq \frac{2\alpha_{s2}}{n-2}+\alpha_{s1}-\left(N_{RL}(i)+N_{LL}(i)\right)
    \geq \frac{2\alpha_{s2}}{n-2}+\alpha_{s1}-\abs{L}.
  \end{multline*}
  Substituting the values for $\alpha_{s1}$ and $\alpha_{s2}$ we obtain for even $n$,
  \[\frac{2}{n-2}\left(\alpha_{s2}+\sum_{j\in I(i)}\alpha_{j2}\right)\geq\frac16(n-1)(n-6)+\frac{n}{2}-\frac{n-2}{2}=\frac16(n-1)(n-6)+1=\frac{(n-4)(n-3)}{6} >0,\]
  and for odd $n$,
  \[\frac{2}{n-2}\left(\alpha_{s2}+\sum_{j\in I(i)}\alpha_{j2}\right)\geq\frac16(n-1)(n-3)+1-\frac{n-1}{2}=\frac16(n-1)(n-3)-\frac{n-3}{2}=\frac{(n-4)(n-3)}{6}>0.\qedhere\]
\end{proof}

\begin{theorem}\label{thm:char_1}
   For $n\geq 5$, the basis
  encoded by $(s,\pi,L,U)$ is feasible if and only if the following conditions are satisfied:
  \begin{enumerate}[(i)]
  \item $\lvert U\rvert=\lfloor (n-2)/2\rfloor$ and $\lvert L\rvert=\lceil (n-2)/2\rceil$,\label{item:UL_card}
  \item $\pi(s)\in U$,\label{item:pi(s)}
  \item For every $i\in[n]\setminus\{s,s-1\}$, $\displaystyle \alpha_{s1}+\sum_{j\in I(i)} \alpha_{j1}\geq0$.\label{item:nonneg}
  \end{enumerate}
\end{theorem}
\begin{proof}
  Suppose that the basis encoded by $(s,\pi,L,U)$ is feasible. Conditions~(\ref{item:UL_card})
  and~(\ref{item:pi(s)}) follow from Lemmas~\ref{lem:cardinality_LU}
  and~\ref{lem:nodes_of_long_path}, respectively, and Condition~(\ref{item:nonneg}) from
  Lemmas~\ref{lem:cardinality_LU} and~\ref{lem:asymptotic_eta} together with $\eta_i\geq 0$. Conversely, suppose that the three conditions
  in the theorem are satisfied. By Lemma~\ref{lem:characterisation_eta}, it is sufficient to verify
  $\beta^{n-1}\leq \xi_s \leq \beta$ and $\eta_i \geq 0$ for all $i \in [n]\setminus\{s-1\}$. The
  first of these conditions follows immediately from Lemma~\ref{lem:value_eta_s}, and the
  second one from~(\ref{item:nonneg}) together with Lemmas~\ref{lem:asymptotic_eta}
  and~\ref{lem:second_order}.
\end{proof}

Next, we simplify Theorem~\ref{thm:char_1}\eqref{item:nonneg} and provide another characterization
of feasible bases in Theorems~\ref{thm:oddn_char} and~\ref{thm:evenn_char}. To do this, by using
Lemmas~\ref{lem:cardinality_LU} and~\ref{lem:nodes_of_long_path} we deduce a restriction for the
possible values of $s-1$ (Lemma~\ref{lem:s-1_value}). Furthermore, in
Lemmas~\ref{lem:cardinality_XZ} and~\ref{lem:Z_dom_X}, we deduce conditions on the number of
different types of arcs in various parts of the good augmented tree, where the type of an arc is
determined by where it starts (in $L$, $U$, or $R$) and where it ends (in $L$, $U$ or $R$).
\begin{lemma}\label{lem:s-1_value}
  If $(s,\pi,L,U)$ encodes a feasible basis, then $s-1\in R$ if $n$ is odd, and $s-1\in R\cup U$ if $n$
  is even.  
\end{lemma}
\begin{proof}
  By Lemmas~\ref{lem:cardinality_LU} and~\ref{lem:nodes_of_long_path}, $\xi_s=1-\alpha_{s1}\delta+O\left(\delta^2\right)$, and
  then~\eqref{eq:simplified_constraints_0} implies
  $\xi_{s-1}=1-\left(\alpha_{s1}-1\right)\delta+O\left(\delta^2\right)$.
  If $s-1\in L$ then~\eqref{eq:simplified_constraints_2b} implies
  \[\eta_{s-2}=\beta^{n-1}-\xi_{s-1}=1-(n-1)\delta-1+\left(\alpha_{s1}-1\right)\delta+O\left(\delta^2\right)=\left(\alpha_{s1}-n\right)\delta+O\left(\delta^2\right)<0.\]
  If $n$ is odd then $\alpha_{s1}=1$, and for $s-1\in U$,~\eqref{eq:simplified_constraints_2c} implies
  \[\eta_{s-2}=\beta-\xi_{s-1}=1-\delta-1+O\left(\delta^2\right)=-\delta+O\left(\delta^2\right)<0.\qedhere\]  
\end{proof}
\begin{lemma}\label{lem:cardinality_XZ}
  If $\pi(s)\in U$ and $s-1\in R \cup U$ then $\abs{W_{UL}}+\abs{W_{UR}}+\abs{W_{RL}}= \abs{W_{LU}}+1$.
\end{lemma} 
\begin{proof}
  We start with
  $U=\left(U\cap\pi^{-1}(U)\right)\cup\left(U\cap\pi^{-1}(L)\}\right)\cup\left(U\cap\pi^{-1}(R)\}\right)\cup\left(U\cap\{s-1\}\right)$
  which implies
  \begin{equation}\label{eq:U_1}
    \abs{U}=\abs{W_{UU}}+\abs{W_{UL}}+\abs{W_{UR}}+
    \begin{cases}
      1 &\text{if }s-1\in U,\\
      0 &\text{if }s-1\in R.
    \end{cases}
  \end{equation}
  On the other hand, $U=\left(U\cap\pi(U)\right)\cup\left(U\cap\pi(L)\}\right)\cup\left(U\cap\pi(R)\}\right)\cup\left(U\cap\pi(\{s\})\right)$,
  hence
  \begin{equation}\label{eq:U_2}
    \abs{U}=\abs{W_{UU}}+\abs{W_{LU}}+\abs{W_{RU}}+1.
  \end{equation}
  From~\eqref{eq:U_1} and~\eqref{eq:U_2}, we obtain
  \begin{equation}\label{eq:double_counting}
    \abs{W_{UL}}+\abs{W_{UR}}=\abs{W_{LU}}+\abs{W_{RU}}+
    \begin{cases}
      1 &\text{if }s-1\in R,\\
      0 &\text{if }s-1\in U. 
    \end{cases}
  \end{equation}
  Now
  \[
    \abs{W_{UL}}+\abs{W_{UR}}+\abs{W_{RL}}\stackrel{\eqref{eq:double_counting}}{=}\abs{W_{LU}}+\abs{W_{RU}}+\abs{W_{RL}}+
    \begin{cases}
      1 &\text{if }s-1\in R\\
      0 &\text{if }s-1\in U 
    \end{cases}
    =\abs{W_{LU}}+1,\]  
  where the last equality comes from the observation that $W_{RU}\cup W_{RL}=\emptyset$ if $s-1=1$
  and $W_{RU}\cup W_{RL}=\{1\}$ if $s-1\in U\cup L$.
\end{proof}
\begin{remark}\label{rem:i*}
  Lemma~\ref{lem:cardinality_XZ} implies that if $s-1\in R \cup U$ and $\pi(s) \in U$, then there
  exists a unique index $i^*=i^*(\pi,L,U)\in[n]\setminus\{s,s-1\}$ such that
  $N_{UL}(i^*)+N_{UR}(i^*)+N_{RL}(i^*)= N_{LU}(i^*)+1$ and
  $N_{UL}(i)+N_{UR}(i)+N_{RL}(i)\leq N_{LU}(i)$ for all $i\in I(i^*-1)$.
	
\end{remark}
\begin{lemma}\label{lem:Z_dom_X}
  Suppose the quadruple $(s,\pi,L,U)$ encodes a feasible basis, and let $i^*= i^*(\pi,L,U)$. Then 
  \begin{enumerate}[(i)]
  \item If $n$ is odd, then $i^*=n$.
  \item If $n$ is even, then
    \begin{itemize}
    \item $N_{UL}(i)+N_{UR}(i)+N_{RL}(i)\leq N_{LU}(i)+1$ for all $i\in[n]\setminus\{s,s-1\}$, and
    \item $\abs{(I(i^*)\setminus\{1,\pi^{-1}(1)\}} \geq (n-4)/2$.
    \end{itemize}
  \end{enumerate}  
\end{lemma}
\begin{proof}
  Theorem~\ref{thm:char_1}(\ref{item:nonneg}) together with~\eqref{alpha_values} implies
  \begin{multline*}\label{eq1:lem_Z_dom_X}
    0\leq\sum_{j\in
      I(i^*)}\alpha_{j1}=\alpha_{s1}+N_{UU}(i^*)+N_{LL}(i^*)+(3-n)N_{UL}(i^*)+(n-1)N_{LU}(i^*)
    \\+(2-n)\left(N_{UR}(i^*)+N_{RL}(i^*)\right),
  \end{multline*}
  hence
  \begin{multline*}
    0= \alpha_{s1}+ N_{UU}(i^*) + N_{LL}(i^*) + N_{UL}(i^*) +
    N_{LU}(i^*) +(n-2)\left[N_{LU}(i^*)-N_{UL}(i^*)-N_{UR}(i^*)- N_{RL}(i^*)\right]\\
    =\alpha_{s1}+ N_{UU}(i^*) + N_{LL}(i^*) + N_{UL}(i^*) +
    N_{LU}(i^*) -(n-2)\\
    = \alpha_{s1}+\abs{I(i^*)\setminus\{1,\pi^{-1}(1)\}}-(n-2),
  \end{multline*}
  and therefore,
  \begin{equation}\label{eq3:lem_Z_dom_X}
    \abs{I(i^*)\setminus\{1,\pi^{-1}(1)\}}\geq n-2-\alpha_{s1}=\begin{cases}
      n-3  &\text{if $n$ is odd},\\
      (n-4)/2 &\text{if $n$ is even}.
    \end{cases}
  \end{equation}
  If $n$ is odd then $\abs{I(i)\setminus\{1,\pi^{-1}(1)\}}\leq\abs{I(i)}\leq i-2$ for all
  $i\in [n]\setminus\{s,s-1\}=\{3,4,\dots,n\}$, and together with $\pi^{-1}(1)\in I(n-1)$, we deduce
  $i^*=n$. Now let $n$ be even, and assume that there exists $i\in I(s-2)\setminus I(i^*)$ such that
  $N_{UL}(i)+N_{UR}(i)+N_{RL}(i)\geq N_{LU}(i)+2$.
  As before we apply Theorem~\ref{thm:char_1}(\ref{item:nonneg}) together with~\eqref{alpha_values}
  and $\alpha_{s1}=n/2$:
  \begin{multline*}
    0\leq n/2+ N_{UU}(i) + N_{LL}(i) + N_{UL}(i) + N_{LU}(i)
    +(n-2)\left[N_{LU}(i)-N_{UL}(i)-N_{UR}(i)- N_{RL}(i)\right]\\
    \leq n/2+ N_{UU}(i) + N_{LL}(i) + N_{UL}(i) + N_{LU}(i)
    -2(n-2)\leq -3n/2+4+(n-3)=1-n/2<0,
  \end{multline*}
  which is the required contradiction.
\end{proof}
In the next two theorems we show that the necessary conditions from
Lemmas~\ref{lem:cardinality_LU},~\ref{lem:nodes_of_long_path},~\ref{lem:s-1_value},
and~\ref{lem:Z_dom_X} are also sufficient.

\begin{theorem}\label{thm:oddn_char}
  Let $n=2k+1$. The quadruple $(s,\pi,L,U)$ encodes a feasible basis if and only if the following
  conditions are satisfied.
  \begin{enumerate}[(i)]
  \item $\lvert U \rvert= k-1$ and $\lvert L \rvert= k$,\label{item:odd_n_cond_1}
  \item $\pi(s)\in U$,\label{item:odd_n_cond_2}
  \item $s-1\in R$,\label{item:odd_n_cond_3}
  \item $i^*(\pi,L,U)=n$.\label{item:odd_n_cond_4}  
  \end{enumerate}
\end{theorem}
\begin{proof}
  If the basis corresponding to $(s,\pi,L,U)$ is feasible then conditions~(\ref{item:odd_n_cond_1})
  to~(\ref{item:odd_n_cond_4}) are implied by
  Lemmas~\ref{lem:cardinality_LU},~\ref{lem:s-1_value},~\ref{lem:nodes_of_long_path},
  and~\ref{lem:Z_dom_X}. For the converse, suppose that the conditions are satisfied. From
  Theorem~\ref{thm:char_1} together with~(\ref{item:odd_n_cond_1})--(\ref{item:odd_n_cond_3}) it
  follows that it is sufficient to verify $\alpha_{21}+\sum_{j\in I(i)}\alpha_{j1}\geq 0$ for all
  $i\in[n]\setminus\{1,2\}$. From Lemma~\ref{lem:cardinality_XZ} with~(\ref{item:odd_n_cond_2})
  and~(\ref{item:odd_n_cond_3}), it follows that $i^*=i^*(\pi,L,U)$ is well
  defined, and from~(\ref{item:odd_n_cond_3}), $W_{RL}=\emptyset$. Using this together
  with~(\ref{item:even_n_cond_4}) and the definition of $i^*$
  \begin{align}
  N_{LU}(i)- N_{UL}(i) - N_{UR}(i) &\geq 0&&\text{for all } i\in[n]\setminus\{1,2,n\}\label{eq1:thm_oddn_char},\\
  N_{LU}(n)- N_{UL}(n) - N_{UR}(n) &= -1.\label{eq2:thm_oddn_char}
  \end{align}
  Now
  \begin{multline*}
    \alpha_{s1}+\sum_{j\in I(i)}\alpha_{j1}= 1+ N_{UU}(i) + N_{LL}(i) + (3-n)N_{UL}(i) + (n-1)N_{LU}(i)  +
    (2-n)N_{UR}(i) \\
    =  1+ N_{UU}(i) + N_{LL}(i) + N_{UL}(i) + N_{LU}(i)
    +(n-2)\left[N_{LU}(i)-N_{UL}(i)-N_{UR}(i)\right]\\
    \stackrel{(\ref{eq1:thm_oddn_char})}{\geq} 1+ N_{UU}(i) + N_{LL}(i) + N_{UL}(i) + N_{LU}(i)\geq 0,
  \end{multline*}
  for all $i\in[n]\setminus\{1,2,n\}$, and
  \begin{multline*}
  \sum_{j\in I(n)}\alpha_{j1}= 1+ N_{UU}(n) + N_{LL}(n) + N_{UL}(n) + N_{LU}(n)
  +(n-2)\left[N_{LU}(n)-N_{UL}(n)-N_{UR}(n)\right]\\
  \stackrel{(\ref{eq2:thm_oddn_char})}{=}3-n+ N_{UU}(n) +
  N_{LL}(n) + N_{LU}(n) + N_{UL}(n)=
  3-n+\abs{I(n)\setminus\{\pi^{-1}(1),2\}}=0 .
  \end{multline*}
  where the last equality is a consequence of $\abs{I(n)\setminus\{\pi^{-1}(1),s\}}=n-3$. 
\end{proof}
\begin{theorem}\label{thm:evenn_char}
  Let $n=2k$ and let $i^*=i^*(\pi,L,U)$. The quadruple $(s,\pi,L,U)$ encodes a feasible basis if and only if the following conditions are satisfied.
  \begin{enumerate}[(i)]
  \item $\lvert U \rvert= \lvert L \rvert=k-1$,\label{item:even_n_cond_1}
  \item $\pi(s)\in U$,\label{item:even_n_cond_2}
  \item $s-1 \in R\cup U$,\label{item:even_n_cond_3}
  \item $\abs{I(i^*)\setminus\{1,\pi^{-1}(1)\}} \geq (n-4)/2$,\label{item:even_n_cond_5}
  \item $N_{UL}(i)+N_{UR}(i)+N_{RL}(i)\leq N_{LU}(i)+1$ for all $i \in I(s-2)\setminus I(i^*)$.\label{item:even_n_cond_4}  
  \end{enumerate}
\end{theorem}
\begin{proof}
  If the basis corresponding to $(s,\pi,L,U)$ is feasible then conditions~(\ref{item:even_n_cond_1})
  to~(\ref{item:even_n_cond_5}) are implied by Lemmas~\ref{lem:cardinality_LU},~\ref{lem:s-1_value}
  and~\ref{lem:nodes_of_long_path},~\ref{lem:Z_dom_X}, and~\ref{lem:Z_dom_X}. For the converse,
  suppose that the conditions are satisfied. From Theorem~\ref{thm:char_1} together
  with~(\ref{item:even_n_cond_1}) and~(\ref{item:even_n_cond_2}) it follows that it is sufficient to
  verify $\alpha_{s1}+\sum_{j\in I(i)}\alpha_{j1}\geq 0$ for all $i\in[n]\setminus\{s,s-1\}$. From
  Lemma~\ref{lem:cardinality_XZ} with~(\ref{item:even_n_cond_2}) and~(\ref{item:even_n_cond_3}),
  $i^*$ is well defined. Using~(\ref{item:even_n_cond_4}) and the definition of $i^*$
  \begin{align}
  N_{LU}(i)- N_{UL}(i) - N_{UR}(i) - N_{RL}(i) &\geq 0&&\text{for all } i\in I(i^*-1),\label{eq1:thm_evenn_char}\\
  N_{LU}(i)- N_{UL}(i) - N_{UR}(i) - N_{RL}(i) &\geq -1&&\text{for all } i\in I(s-2)\setminus I(i^*-1).\label{eq2:thm_evenn_char}
  \end{align}
  Now 
  \begin{multline*}
  \alpha_{s1}+\sum_{j\in I(i)}\alpha_{j1}=  n/2+ N_{UU}(i) + N_{LL}(i) + (3-n)N_{UL}(i) + (n-1)N_{LU}(i)  +
  (2-n)\left[N_{UR}(i)+ N_{RL}(i)\right]\\
   = n/2+ N_{UU}(i) + N_{LL}(i) + N_{UL}(i) + N_{LU}(i)
  +(n-2)\left[N_{LU}(i)-N_{UL}(i)-N_{UR}(i)- N_{RL}(i)\right]\\ 
  \stackrel{(\ref{eq1:thm_evenn_char}),(\ref{eq2:thm_evenn_char})}{\geq}\begin{cases}
  n/2+ N_{UU}(i) + N_{LL}(i) + N_{UL}(i) + N_{LU}(i) \geq 0&\text{for all } i\in I(i^*-1),\\
  n/2+ 2-n + N_{UU}(i) + N_{LL}(i) + N_{UL}(i) + N_{LU}(i)\stackrel{(\ref{item:even_n_cond_5})}{\geq} 0&\text{for all } i\in I(s-2)\setminus I(i^*-1). \qedhere
  \end{cases}  
  \end{multline*}
\end{proof}

We illustrate Theorems~\ref{thm:char_1},~\ref{thm:oddn_char}, and~\ref{thm:evenn_char} by checking
the feasibility of the three bases represented in Figure~\ref{fig:three_bases}.
\begin{figure}[htb]
  \begin{subfigure}[b]{.32\linewidth}
		\centering 
		\begin{tikzpicture}[scale=2,
		U/.style={rectangle,draw,fill=black,inner sep=0pt,outer sep=1pt,minimum size=2mm},
		L/.style={rectangle,draw,fill=none,inner sep=0pt,outer sep=1pt,minimum size=2mm},
		base/.style={circle,draw,fill=black,inner sep=0pt,outer sep=1pt,minimum size=2mm},
		start/.style={circle,draw,fill=none,inner sep=0pt,outer sep=1pt,minimum size=2mm}]
		\node[start,label={[label distance=-0cm]90:$1$}] (v1) at (90:1) {};
		\node[L,label={[label distance=-.1cm]30:$5$}] (v2) at (30:1) {};
		\node[U,label={[label distance=-.1cm]-30:$2$}] (v3) at (-30:1) {};
		\node[L,label={[label distance=-0cm]-90:$3$}] (v4) at (-90:1) {};
		\node[base,label={[label distance=-.1cm]-150:$4$}] (v5) at (-150:1) {};
		\node[U,label={[label distance=-.1cm]-210:$6$}] (v6) at (-210:1) {};    
		\draw[very thick,->] (v1) -- (v2);
		\draw[very thick,->] (v2) -- (v3);
		\draw[very thick,->] (v3) -- (v1);
		\draw[very thick,->] (v4) -- (v5);
		\draw[very thick,->] (v5) -- (v6);
		\draw[very thick,->] (v6) -- (v4);    
		\draw[bend right=20,->] (v1) to (v4);
		\draw[bend right=0,->] (v2) to (v6);
		\draw[bend right=30,->] (v3) to (v2);
		\draw[bend right=10,->] (v4) to (v1);
		\draw[bend right=0,->] (v6) to (v3);    
              \end{tikzpicture}
              \caption{$s=4$, $\pi=(152)(346)$,\\ $L=\{3,5\}$, $U=\{2,6\}$.}\label{fig:ex1_graph}
            \end{subfigure}\hfill
            \begin{subfigure}[b]{.32\linewidth}
	\centering 
	\begin{tikzpicture}[scale=2,
	U/.style={rectangle,draw,fill=black,inner sep=0pt,outer sep=1pt,minimum size=2mm},
	L/.style={rectangle,draw,fill=none,inner sep=0pt,outer sep=1pt,minimum size=2mm},
	base/.style={circle,draw,fill=black,inner sep=0pt,outer sep=1pt,minimum size=2mm},
	start/.style={circle,draw,fill=none,inner sep=0pt,outer sep=1pt,minimum size=2mm}]
	\node[start,label={[label distance=-.1cm]90:$1$}] (v1) at (90:1) {};
	\node[base,label={[label distance=-.1cm]38.6:$2$}] (v2) at (38.6:1) {};
	\node[U,label={[label distance=-.1cm]347.1:$7$}] (v3) at (347.1:1) {};
	\node[L,label={[label distance=-.1cm]295.7:$3$}] (v4) at (295.7:1) {};
	\node[U,label={[label distance=-.1cm]244.3:$5$}] (v5) at (244.3:1) {};
	\node[L,label={[label distance=-.1cm]192.9:$4$}] (v6) at (192.9:1) {};
	\node[L,label={[label distance=-.1cm]141.4:$6$}] (v7) at (141.4:1) {};
	\draw[very thick,->] (v1) -- (v2);
	\draw[very thick,->] (v2) -- (v3);
	\draw[very thick,->] (v3) -- (v4);
	\draw[very thick,->] (v4) -- (v5);
	\draw[very thick,->] (v5) -- (v6);
	\draw[very thick,->] (v6) -- (v7);
	\draw[very thick,->] (v7) -- (v1);
	\draw[bend right=10,->] (v1) to (v4);
	\draw[bend right=10,->] (v3) -- (v1);
	\draw[bend right=30,->] (v4) to (v3);
	\draw[bend right=10,->] (v5) to (v7);
	\draw[bend right=30,->] (v6) to (v5);
	\draw[bend right=30,->] (v7) to (v6);    
      \end{tikzpicture}
      \caption{$s=2$, $\pi=(1273546)$,\\ $L=\{3,4,6\}$, $U=\{5,7\}$.}\label{fig:ex2_graph}
      \end{subfigure}\hfill
      \begin{subfigure}[b]{.32\linewidth}
	\centering 
	\begin{tikzpicture}[scale=2,
	U/.style={rectangle,draw,fill=black,inner sep=0pt,outer sep=1pt,minimum size=2mm},
	L/.style={rectangle,draw,fill=none,inner sep=0pt,outer sep=1pt,minimum size=2mm},
	base/.style={circle,draw,fill=black,inner sep=0pt,outer sep=1pt,minimum size=2mm},
	start/.style={circle,draw,fill=none,inner sep=0pt,outer sep=1pt,minimum size=2mm}]
	\node[start,label={[label distance=-.1cm]90:$1$}] (v1) at (90:1) {};
	\node[L,label={[label distance=-.1cm]45:$6$}] (v2) at (45:1) {};
	\node[U,label={[label distance=-.1cm]0: $4$}] (v3) at (0:1) {};
	\node[base,label={[label distance=-.1cm]-45:$5$}] (v4) at (-45:1) {};
	\node[U,label={[label distance=-.1cm]-90:$3$}] (v5) at (-90:1) {};
	\node[L,label={[label distance=-.1cm]-135:$2$}] (v6) at (-135:1) {};
	\node[U,label={[label distance=-.1cm]180:$8$}] (v7) at (180:1) {};
	\node[L,label={[label distance=-.1cm]135:$7$}] (v8) at (135:1) {};
	\draw[very thick,->] (v1) -- (v2);
	\draw[very thick,->] (v2) -- (v3);
	\draw[very thick,->] (v3) -- (v4);
	\draw[very thick,->] (v4) -- (v5);
	\draw[very thick,->] (v5) -- (v1);
	\draw[very thick,->] (v6) -- (v7);
	\draw[very thick,->] (v7) -- (v8);
	\draw[very thick,->] (v8) -- (v6);
	\draw[bend right=0,->] (v1) to (v8);
	\draw[bend right=0,->] (v2) to (v5);
	\draw[bend right=-10,->] (v3) to (v1);
	\draw[bend right=20,->] (v5) to (v7);
	\draw[bend right=0,->] (v6) to (v2);
	\draw[bend right=30,->] (v7) to (v6);
	\draw[bend right=20,->] (v8) to (v3);      
      \end{tikzpicture}
      \caption{$s=5$, $\pi=(16453)(287)$,\\ $L=\{2,6,7\}$, $U=\{3,4,8\}$.}\label{fig:ex3_graph}
\end{subfigure}
  \caption{Three bases for $n=6$, $n=7$ and $n=8$, respectively.}
  \label{fig:three_bases}
\end{figure}
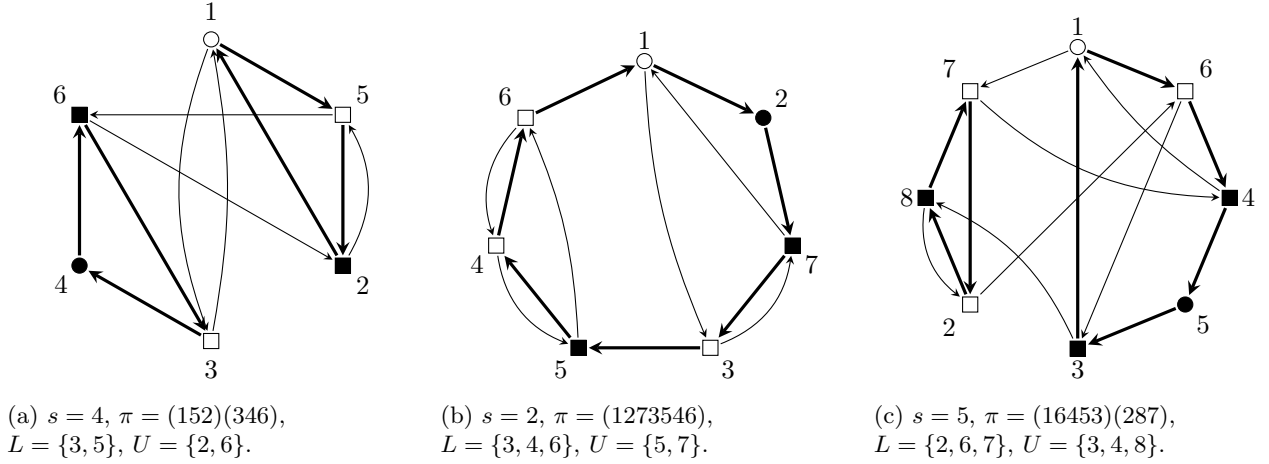
It is immediately clear that all three examples satisfy the first two conditions in
Theorem~\ref{thm:char_1}: the number of empty squares is at least the number of filled squares and
at most the number of filled squares plus one, and the arc leaving the filled circle ends in a
filled square. 
\begin{example}[Figure~\ref{fig:three_bases}(\subref{fig:ex1_graph})]
  We have $5\in W_{LU}$, $6\in W_{UL}$, $1\in W_{RL}$, $2\in W_{UR}$, and therefore
  \[\alpha_{41}+\alpha_{51}+\alpha_{61}+\alpha_{11}+\alpha_{21}=3+5-3-4-4=-3<0,\]
  which implies that the basis is infeasible. The fact that a negative partial sum occurs in the
  very end corresponds to the fact that the third condition in Theorem~\ref{thm:evenn_char} is
  violated as $3=s-1\in L$.
\end{example}
\begin{example}[Figure~\ref{fig:three_bases}(\subref{fig:ex2_graph})]
  Using $3\in W_{LU}$, $4\in W_{LL}$, $5\in W_{UL}$, $6\in W_{LR}$ and $7\in W_{UL}$, we obtain that
  all partial sums in $\alpha_{21}+\alpha_{31}+\dots+\alpha_{71}=1+6+1-4+0-4=0$ are non-negative,
  and the basis is feasible. We get the same result by verifying the conditions in
  Theorem~\ref{thm:oddn_char}:
\begin{enumerate}[(i)]
\item There are $k-1=2$ filled squares and $k=3$ empty squares.
\item The arc leaving the filled circle ends in a filled square.
\item There is a thick arc from node $1$ to node $2$.
\item With $f(i)=N_{UL}(i)+N_{UR}(i)+N_{RL}(i)-N_{LU}(i)$, we have
  $(f(3),f(4),f(5),f(6),f(7))=(-1,-1,0,0,1)$, hence $i^*=7$.
\end{enumerate}
\end{example}
\begin{example}[Figure~\ref{fig:three_bases}(\subref{fig:ex3_graph})]
  Using $6\in W_{LU}$, $7\in W_{LL}$, $8\in W_{UL}$, $1\in W_{RL}$, $2\in W_{LU}$ and $3\in W_{UR}$,
  we obtain that all partial sums in $\alpha_{51}+\alpha_{61}+\alpha_{71}=4+7+1-5-6+7-6=2$ are
  non-negative, so the basis is feasible. We get the same result by verifying the conditions in
  Theorem~\ref{thm:evenn_char}:
\begin{enumerate}[(i)]
\item There are $k-1=3$ filled squares and $k-1=3$ empty squares.
\item The thick arc leaving the filled circle ends in a filled square.
\item The thick arc entering the filled circle starts in a filled square or the empty circle.
\end{enumerate}
With $f(i)=N_{UL}(i)+N_{UR}(i)+N_{RL}(i)-N_{LU}(i)$, we have
$(f(6),f(7),f(8),f(1),f(2),f(3))=(-1,-1,0,1,0,1)$, hence $i^*=1$, and then
\begin{enumerate}[(i)]
  \setcounter{enumi}{3}
\item $\abs{I(1)\setminus\{1,\pi^{-1}(1)\}}=\abs{\{6,7,8\}}=3\geq (8-4)/2=2$, and
\item $f(i)\leq 1$ for all $i\in\{6,7,8,1,2,3\}$.
\end{enumerate}
\end{example}

For small $n$, we can use the characterizations given in Theorems~\ref{thm:char_1},~\ref{thm:oddn_char}
and~\ref{thm:evenn_char} to count feasible bases of the type considered in this section. For $n=5$,
this can be done by hand: By Theorem~\ref{thm:oddn_char}, $s=2$ and $\pi(2)\in U$, in particular
$\pi(2)\neq 1$. Together with $\pi(i)\neq i$ for all $i$, and $\pi(i)\neq i+1$ for all $i\neq 1$,
this leaves only two permutations: $\pi_1=(12543)$ and $\pi_2=(124)(35)$. Then $U$ is fixed because
$\abs{U}=1$ and $\pi(2)\in U$, hence $U=\{5\}$ for $\pi_1$ and $U=\{4\}$ for $\pi_2$. For $\pi_1$
this implies $3\in W_{LR}$, $4\in W_{LL}$ and $5\in W_{UL}$, and the basis is feasible because the
partial sums of $1+0+1-2$ are all non-negative. For $\pi_2$, $3\in W_{LL}$ and $4\in W_{UR}$, so the
basis is infeasible as $1+1-3<0$. For $n\geq 6$, we implemented an algorithm that runs through all
combinations $(s,\pi,U,L)$ and checks the feasibility conditions, using
Theorems~\ref{thm:char_1},~\ref{thm:oddn_char} and~\ref{thm:evenn_char} to limit the search
space. The results for $n\in\{6,\dots,10\}$ are presented in Tables~\ref{tab:n_6}
to~\ref{tab:n_10}. Here the columns (except the first, and for even $n$ the last) correspond to the
values of $s$, and the rows correspond to cycle types, where a tuple $(l_1,\dots,l_t)$ indicates
that $\pi$ has $t$ cycles with lengths $l_1,\dots,l_t$, and $l_1$ is the length of the cycle
containing node $1$. In particular, the first row corresponds to the quasi-Hamiltonian feasible
bases. The shaded cells indicate the number of quasi-Hamiltonian feasible bases and
the total number of feasible bases.
\begin{table}[htb]
  \begin{minipage}{.6\textwidth}
  \centering
  \caption{Numbers of feasible bases for $n=6$.} \label{tab:n_6}
  \begin{tabular}{crrrrrr} \toprule
    &\multicolumn{5}{c}{$s$} & \\ 
    &\multicolumn{1}{c}{2}&\multicolumn{1}{c}{3}&\multicolumn{1}{c}{4}&\multicolumn{1}{c}{5}&\multicolumn{1}{c}{6}& total\\ \cmidrule(lr){1-1}\cmidrule(lr){2-6}\cmidrule(lr){7-7}
$(6)$ & \num{17} & \num{5} & \num{6} & \num{5} & \num{3} & \cellcolor[gray]{0.9}{\num{36}}\\
$(4,2)$ & \num{6} & \num{5} & \num{5} & \num{6} & \num{3} & \num{25}\\
$(3,3)$ & \num{3} & \num{2} & \num{1} & \num{1} & \num{0} & \num{7}\\
$(2,4)$ & \num{0} & \num{2} & \num{1} & \num{1} & \num{0} & \num{4}\\
$(2,2,2)$ & \num{0} & \num{3} & \num{1} & \num{2} & \num{1} & \num{7}\\ \midrule
total & \num{26} & \num{17} & \num{14} & \num{15} & \num{7} & \cellcolor[gray]{0.9}{\num{79}} \\ \bottomrule
  \end{tabular}
\end{minipage}\hfill
\begin{minipage}{.38\textwidth}
  \centering
  \caption{Numbers of feasible bases for $n=7$.}\label{tab:n_7}
  \begin{tabular}{cr} \toprule
    & \multicolumn{1}{c}{$s=2$} \\ \midrule
$(7)$ & \cellcolor[gray]{0.9}{\num{35}} \\
$(5,2)$ & \num{18} \\
$(4,3)$ & \num{11} \\
$(3,4)$ & \num{7} \\
$(3,2,2)$ & \num{3} \\ \midrule
total & \cellcolor[gray]{0.9}{\num{74}} \\ \bottomrule
  \end{tabular}
\end{minipage}
\end{table}
\begin{table}[htb]
  \begin{minipage}{.6\textwidth}
  \centering
  \caption{Numbers of feasible bases for $n=8$.} \label{tab:n_8}
  \begin{tabular}{crrrrrrrr} \toprule
    &\multicolumn{7}{c}{$s$} & \\ 
    &\multicolumn{1}{c}{2}&\multicolumn{1}{c}{3}&\multicolumn{1}{c}{4}&\multicolumn{1}{c}{5}&\multicolumn{1}{c}{6}&\multicolumn{1}{c}{7}&\multicolumn{1}{c}{8}& total\\ \cmidrule(lr){1-1}\cmidrule(lr){2-8}\cmidrule(lr){9-9}
    $(8)$ & \num{1026} & \num{458} & \num{460} & \num{422} & \num{418} & \num{338} & \num{219} & \cellcolor[gray]{0.9}{\num{3341}}\\
$(6,2)$ & \num{516} & \num{333} & \num{339} & \num{323} & \num{330} & \num{263} & \num{169} & \num{2273}\\
$(5,3)$ & \num{337} & \num{145} & \num{134} & \num{121} & \num{115} & \num{106} & \num{80} & \num{1038}\\
$(4,4)$ & \num{247} & \num{120} & \num{110} & \num{113} & \num{108} & \num{76} & \num{52} & \num{826}\\
$(3,5)$ & \num{178} & \num{116} & \num{88} & \num{92} & \num{82} & \num{54} & \num{28} & \num{638}\\
$(2,6)$ & \num{0} & \num{119} & \num{75} & \num{72} & \num{71} & \num{70} & \num{28} & \num{435}\\
$(4,2,2)$ & \num{128} & \num{112} & \num{110} & \num{110} & \num{117} & \num{81} & \num{55} & \num{713}\\
$(2,4,2)$ & \num{0} & \num{115} & \num{83} & \num{83} & \num{84} & \num{74} & \num{34} & \num{473}\\
$(3,3,2)$ & \num{152} & \num{128} & \num{104} & \num{108} & \num{99} & \num{74} & \num{47} & \num{712}\\
$(2,3,3)$ & \num{0} & \num{38} & \num{19} & \num{17} & \num{16} & \num{22} & \num{11} & \num{123}\\
$(2,2,2,2)$ & \num{0} & \num{27} & \num{22} & \num{22} & \num{26} & \num{17} & \num{10} &
                                                                                          \num{124}\\ \midrule
total & \num{2584} & \num{1711} & \num{1544} & \num{1483} & \num{1466} & \num{1175} & \num{733} & \cellcolor[gray]{0.9}{\num{10696}} \\ \bottomrule
  \end{tabular}
\end{minipage}\hfill
\begin{minipage}{.38\textwidth}
  \centering
  \caption{Numbers of feasible bases for $n=9$.}\label{tab:n_9}
  \begin{tabular}{cr} \toprule
    & \multicolumn{1}{c}{$s=2$} \\ \midrule
    $(9)$ & \cellcolor[gray]{0.9}{\num{3891}} \\
$(7,2)$ & \num{1932} \\
$(6,3)$ & \num{1294} \\
$(5,4)$ & \num{954} \\
$(4,5)$ & \num{788} \\
$(3,6)$ & \num{490} \\
$(5,2,2)$ & \num{468} \\
$(4,3,2)$ & \num{651} \\
$(3,4,2)$ & \num{357} \\
$(3,3,3)$ & \num{159} \\
$(3,2,2,2)$ & \num{56} \\ \midrule
total & \cellcolor[gray]{0.9}{\num{11040}} \\ \bottomrule
  \end{tabular}
\end{minipage}
\end{table}

\begin{table}[htb]
  \centering
  \caption{Numbers of feasible bases for $n=10$.} \label{tab:n_10}
  \begin{tabular}{crrrrrrrrrr} \toprule
        &\multicolumn{9}{c}{$s$} & \\ 
    &\multicolumn{1}{c}{2}&\multicolumn{1}{c}{3}&\multicolumn{1}{c}{4}&\multicolumn{1}{c}{5}&\multicolumn{1}{c}{6}&\multicolumn{1}{c}{7}&\multicolumn{1}{c}{8}&\multicolumn{1}{c}{9}&\multicolumn{1}{c}{10}& total\\ \cmidrule(lr){1-1}\cmidrule(lr){2-10}\cmidrule(lr){11-11}
    $(10)$ & \num{163701} & \num{83664} & \num{79720} & \num{74812} & \num{72468} & \num{69116} & \num{58400} & \num{50696} & \num{36127} & \cellcolor[gray]{0.9}{\num{688704}}\\
$(8,2)$ & \num{81890} & \num{56040} & \num{53980} & \num{51041} & \num{49613} & \num{47595} & \num{40438} & \num{34783} & \num{24939} & \num{440319}\\
$(7,3)$ & \num{55099} & \num{27526} & \num{25563} & \num{23849} & \num{23073} & \num{22028} & \num{18711} & \num{16537} & \num{12276} & \num{224662}\\
$(6,4)$ & \num{40832} & \num{21170} & \num{20070} & \num{18836} & \num{18200} & \num{17624} & \num{14926} & \num{12958} & \num{9268} & \num{173884}\\
$(5,5)$ & \num{32416} & \num{16844} & \num{15986} & \num{15197} & \num{14749} & \num{13744} & \num{11212} & \num{9870} & \num{7110} & \num{137128}\\
$(4,6)$ & \num{27019} & \num{14178} & \num{13004} & \num{12550} & \num{12077} & \num{11491} & \num{9689} & \num{8390} & \num{5851} & \num{114249}\\
$(3,7)$ & \num{20834} & \num{14224} & \num{11289} & \num{10970} & \num{10585} & \num{9816} & \num{7704} & \num{6305} & \num{4088} & \num{95815}\\
$(2,8)$ & \num{0} & \num{14343} & \num{10520} & \num{10046} & \num{9722} & \num{8834} & \num{7518} & \num{7011} & \num{3856} & \num{71850}\\
$(6,2,2)$ & \num{20428} & \num{17493} & \num{16895} & \num{16071} & \num{15672} & \num{15245} & \num{13041} & \num{11121} & \num{7987} & \num{133953}\\
$(2,6,2)$ & \num{0} & \num{11948} & \num{9180} & \num{8849} & \num{8592} & \num{7871} & \num{6843} & \num{6247} & \num{3621} & \num{63151}\\
$(5,3,2)$ & \num{27320} & \num{18625} & \num{17596} & \num{16720} & \num{16241} & \num{15256} & \num{12735} & \num{11215} & \num{8349} & \num{144057}\\
$(3,5,2)$ & \num{14514} & \num{12835} & \num{10735} & \num{10524} & \num{10258} & \num{9461} & \num{7399} & \num{6148} & \num{4185} & \num{86059}\\
$(2,5,3)$ & \num{0} & \num{7619} & \num{5464} & \num{5217} & \num{5082} & \num{4527} & \num{3723} & \num{3592} & \num{2083} & \num{37307}\\
$(4,4,2)$ & \num{20235} & \num{14289} & \num{13344} & \num{12873} & \num{12409} & \num{12060} & \num{10245} & \num{8775} & \num{6223} & \num{110453}\\
$(2,4,4)$ & \num{0} & \num{3637} & \num{2662} & \num{2527} & \num{2441} & \num{2277} & \num{1979} & \num{1810} & \num{1010} & \num{18343}\\
$(4,3,3)$ & \num{9194} & \num{4609} & \num{3994} & \num{3814} & \num{3674} & \num{3497} & \num{2972} & \num{2674} & \num{2025} & \num{36453}\\
$(3,4,3)$ & \num{12247} & \num{8298} & \num{6394} & \num{6180} & \num{5943} & \num{5584} & \num{4415} & \num{3683} & \num{2529} & \num{55273}\\
$(4,2,2,2)$ & \num{3354} & \num{3504} & \num{3324} & \num{3219} & \num{3124} & \num{3118} & \num{2675} & \num{2254} & \num{1607} & \num{26179}\\
$(2,4,2,2)$ & \num{0} & \num{5365} & \num{4233} & \num{4096} & \num{3987} & \num{3751} & \num{3350} & \num{2959} & \num{1777} & \num{29518}\\
$(3,3,2,2)$ & \num{6061} & \num{6490} & \num{5518} & \num{5430} & \num{5315} & \num{4896} & \num{3879} & \num{3274} & \num{2346} & \num{43209}\\
$(2,3,3,2)$ & \num{0} & \num{3963} & \num{2958} & \num{2842} & \num{2789} & \num{2455} & \num{2045} & \num{1967} & \num{1232} & \num{20251}\\
$(2,2,2,2,2)$ & \num{0} & \num{582} & \num{472} & \num{463} & \num{448} & \num{446} & \num{406} &
                                                                                                  \num{343} & \num{213} & \num{3373}\\ \midrule
total & \num{535144} & \num{367246} & \num{332901} & \num{316126} & \num{306462} & \num{290692} & \num{244305} & \num{212612} & \num{148702} & \cellcolor[gray]{0.9}{\num{2754190}} \\ \bottomrule
  \end{tabular}
\end{table}

In these tables, we observe that the number of feasible bases decreases as the number of cycles
increases, and that for a fixed number of cycles the number of feasible bases is larger if node $1$
is on a long cycle. In view of Conjecture~\ref{con:EFKM}, we are interested in the ratio between the
number of quasi-Hamiltonian bases and the total number of feasible bases, and we would like this
ratio to be bounded below by 1 divided by a polynomial function of $n$. Defining $a_n$ and $b_n$ to
be the numbers of quasi-Hamiltonian feasible bases and the total number of feasible bases in the
considered class, respectively, we can summarize our counting results as shown in Table~\ref{tab:summary}.
\begin{table}[htb]
  \centering
  \caption{The proportion of quasi-Hamiltonian feasible bases in the class of feasible bases
    considered in this section.}
  \label{tab:summary}
  \begin{tabular}{crrr}\toprule
    $n$ & \multicolumn{1}{c}{$a_n$} & \multicolumn{1}{c}{$b_n$} & \multicolumn{1}{c}{$na_n/b_n$} \\ \midrule
    $5$ & \num{1} & \num{1} & \num{5.0000} \\
    $6$ & \num{36} & \num{79} & \num{2.7342} \\
    $7$ & \num{35} & \num{74} & \num{3.3108} \\
    $8$ & \num{3341} & \num{10696} & \num{2.4989} \\
    $9$ & \num{3891} & \num{11040} & \num{3.1720} \\
    $10$ & \num{688704} & \num{2754190} & \num{2.5006} \\
    $11$ & \num{801114} & \num{2884325} & \num{3.0552} \\
    $12$ & \num{234123800} & \num{1113400022} & \num{2.5233} \\
    $13$ & \num{269326587} & \num{1172169769} & \num{2.9870} \\
    \bottomrule
  \end{tabular}  
\end{table}

%%% Local Variables:
%%% mode: latex
%%% TeX-master: "paper"
%%% End: